\documentclass[a4paper,10pt]{article}

\usepackage{amsfonts}
\usepackage{amsmath}
\usepackage{amssymb}
\usepackage[arrow]{xy}
\usepackage{amsthm}

\newtheorem{theorem}{Theorem}[section]

\newtheorem{definition}[theorem]{Definition}

\newtheorem{lemma}[theorem]{Lemma}
\newtheorem{proposition}[theorem]{Proposition}
\newtheorem*{remark}{Remark}

\newcommand{\te}{\textrm}

\def \Br {{\rm{Br\,}}}

\def\ov{\overline}

\def \Z {{\mathbb Z}} 
 
\def \Q {{\mathbb Q}}
 
\def \Br {{\rm Br\,}}

\newcommand{\bthe}{\begin{theo}}
\newcommand{\ble}{\begin{lem}}
\newcommand{\bpr}{\begin{prop}}
\newcommand{\bco}{\begin{cor}}
\newcommand{\bde}{\begin{defi}}
\newcommand{\ethe}{\end{theo}}
\newcommand{\ele}{\end{lem}}
\newcommand{\epr}{\end{prop}}
\newcommand{\eco}{\end{cor}}
\newcommand{\ede}{\end{defi}}

\begin{document}

\begin{center}
{\bf \Large Diagonal quartic surfaces 
and transcendental elements of the Brauer group}
\end{center}

\begin{center}
{\bf  Evis Ieronymou}
\end{center}

\begin{abstract}

\noindent We exhibit central simple algebras over the function field of a diagonal quartic surface over the complex numbers that represent the $2$-torsion part of its Brauer group. We investigate whether the $2$-primary part of the Brauer group of  a diagonal quartic surface
 over a number field is algebraic and give sufficient conditions for this to be the case. In the last section we give an obstruction to weak approximation due to a transendental class on a specific diagonal quartic surface, an obstruction which cannot be explained by the algebraic Brauer group which in this case is just the constant algebras.
\end{abstract}
\section{Introduction}

Let $k$ be a number field and $X/k$ a smooth, projective, geometrically integral variety. We denote by $\Br(X)$ the cohomological Brauer group of $X$, that is $H^2_{\te{\'{e}t}}(X,\mathbb G_m)$. The algebraic Brauer group, denoted $\Br_1(X)$, is defined as the kernel of the natural map $\Br(X)\mapsto \Br(\overline X)$, where $\overline k$ is an algebraic closure of $k$ and $\overline X=X\times_k \overline k$. We call a class in $\Br(X)$ algebraic if it belongs to Br$_1(X)$, transcendental otherwise. Algebraic classes have been used for arithmetical purposes, giving information about the rational points of $X$. One can immediately ask (a) Is $\Br_1(X)=\Br(X)$?
(b) If not, can we use transcendental classes of the Brauer group of $X$ to get extra information about the rational points of $X$?

For curves or rational surfaces the answer to (a) is yes for the trivial reason that $\Br(\overline X)=0$. 
This last fact is in direct contrast with $K3$ surfaces where $\textrm{Br}(\overline{X})\cong(\mathbb{Q}/\mathbb{Z})^{\beta_2-\rho}$ is far from trivial. Here $\beta_2$ is the second Betti number which is $22$ for any $K3$ surface, and $\rho$ is the geometric Picard number, that is the rank of $\textrm{Pic}(\overline X)$, which satisfies $1\leq\rho\leq20$. Therefore for $K3$ surfaces transcendental classes might exist, and will certainly exist after an appropriate extension of the ground field. For $K3$ surfaces we also know that $\textrm{Br}(X) / \textrm{Br}_1(X)$ is finite when the ground field is a  number field (proved by Skorobogatov and Zarhin in \cite{SkoroK3}). 

Generally speaking there are techniques to handle algebraic elements of the Brauer group, at least when we have fine enough information about Pic$(\overline X)$. However there are no such techniques
for transcendental elements. This makes them more mysterious and consequently question (b) more difficult to answer. Indeed examples of transcendental elements are more rare in the literature. Arithmetic applications of transcendental elements are exhibited by Harari in \cite{Har}, where the 
variety considered is a conic bundle over the projective plane. For $K3$ surfaces transendental elements of the Brauer group appear in papers of Wittenberg \cite{Witt}, Skorobogatov and Swinnerton-Dyer \cite{Skorokaiswin}, Harari and Skorobogatov \cite{SkHar}. In these cases an elliptic fibration is used. We also use an elliptic fibration, but in our case the fibration has no section which complicates things a little bit.

In this paper we concentrate on the special subclass of $K3$ surfaces consisting of diagonal quartic surfaces. A good reason for doing this is that they are special enough and adequately studied to be amenable to computations. On the other hand they are wide enough to draw general conclusions or even hints of what general conclusions about $K3$ surfaces should be. Relevant to this paper are the papers of Swinnerton-Dyer \cite{SD} and Bright \cite{Br2}, where diagonal quartic surfaces are studied. In particular by \cite{Br2} we have a good control over the algebraic Brauer group. This leaves open the role played by transendental elements and evidence suggests that they might provide extra arithmetic information. It is worth noting here that we cannot apply the results of \cite{bigpap} since the geometric assumptions on the fibration do not hold in our case.

The structure of the paper is as follows:
In section $3$ we let $S/\mathbb C$ be the diagonal quartic surface $x^4-y^4=z^4-w^4$. The main result is that we exhibit central simple algebras over its function field that represent the $2$-torsion part of $\Br(S)$ (theorem \ref{bigthm}). To prove this we utilise an elliptic fibration of $S$ over the projective line, with generic fibre $C$. The curve $C$ is an element of exact order $2$ in the Weil-Ch\^atelet group of its Jacobian. We use the theory of torsors to exhibit elements that span $\Br(C)[2]$ and then use Grothendieck's purity theorem for the Brauer group to see which of these elements belong to $\Br(S)$.

In sections $4$ and $5$ we investigate whether the $2$-primary part of the Brauer group of  a diagonal quartic surface
 over a number field is algebraic. We give sufficient conditions for this to be the case. 
Section $4$ is concerned with the fixed surface $S$ whereas section $5$ is concerned
 with diagonal quartic surfaces with arbitrary coefficients. The main result of each section is presented at the beginning. Finally, in section $6$ we give an example where a transcendental class gives extra information about rational points. 
In particular we show that a transendental class gives an obstruction to weak approximation on $S/\mathbb Q(i,\sqrt[4]2)$. In this case the algebraic Brauer group is simply the constant algebras, and hence gives no information about the rational points.

\section{Notation and preliminaries}
All cohomology groups are \'{e}tale. When $k$ is a field we denote by $\overline k$ a separable closure of $k$. We fix such a closure throughout. For a $\te{Gal}(\overline k/k)$-module $M$ we denote by  $H^i(k,M)$ the group $H^i(\te{Gal}(\overline k/k),M)$. For a group $M$ we denote by $M[n]$ its $n$-torsion subgroup.

Let $X/k$ be a variety and $k\subseteq L$ a field extension. We denote $X_L$ the variety $X\times_k L/L$. When $L=\overline k$ we simply write $\overline X$. We denote by $k(X)$ the function field of $X$, and by $X(L)$ the set of $L$-valued points of $X$, that is $X(L)=\textrm{Hom}_{\te{spec}(k)}(\te{spec}(L),X)$.

Let $k$ be a field of characteristic coprime to $n$ that contains the $n$-th roots of unity. For the definition and basic properties of the symbol $(a,b)_n$ as an element of $\Br(k)$ we refer the reader to \cite[chapter XIV]{Serre}. 

By a result of Grothendieck $\Br(X)$ canonically injects in the Brauer group of its function field, when $X$ is smooth, proper and irreducible (cf. \cite[Theorem 1.3.2]{CT}). We will assume this implicitly throughout. General results on the Brauer group and residue maps can be found in \cite{Groth}.

For the definitions and basic properties  of torsors, groups of multiplicative type, their dual groups, type of a torsor under a group of multiplicative type, $n$-coverings of abelian varieties we refer the reader to \cite{Skoro}. We use the notation used there.

Next we give two theorems, due to Colliot-Th\'el\`ene and Sansuc, which form the conceptual framework of the work we will be doing.

The first theorem gives a link between torsors under groups of multiplicative type and the Brauer group of a variety.

\begin{theorem}\label{type theorem}
Let $X/k$ be a smooth, projective, geometrically integral variety. Let $G$ be a group of multiplicative type, $ \widehat{G}$ its dual group, $Y/X$ a torsor under $G$ and let $ \widehat{G} \xrightarrow{\lambda} \te{Pic}(\overline{X})$ be the type of $Y/X$.
The following diagram is commutative:

\centerline{
\begin{xy}
(40,30)*+{H^1(X,\widehat{G})}="ka"; 
(40,60)*+{H^1(k,\widehat{G})}="pa"; 
(90,30)*+{\Br_1(X)}="kd";
(90,60)*+{H^1(k,\textrm{Pic}(\overline{X}))}="pd"; 
{\ar@{->}^{\cup [Y/X]} "ka";"kd"};
{\ar@{->}^{\lambda_*} "pa";"pd"};
{\ar@{->}^{s^*} "pa";"ka"};
{\ar@{->}_{\psi} "kd";"pd"};
\end{xy}}

\noindent Here $s$ is the structure morphism of $X$ and $\psi$ is the map coming from the Hochschild-Serre spectral sequence.
\end{theorem}
\begin{proof}
\cite[Theorem 4.1.1]{Skoro}.
\end{proof}
The second theorem uses the theory of local description of torsors under groups of multiplicative type in order to give the data needed to apply the theorem above.
\begin{theorem}\label{cyclic constr}

Let $X/k$ be a smooth, projective, geometrically integral variety.
\begin{enumerate}
 \item Let $f\in k(X)^*$ be a function such that $\textrm{div}(f)=nD$. Then there is a torsor $Y/X$ under $\mu_n$ such that
\begin{enumerate}
 \item type$[Y/X]:\mathbb Z/n\Z\rightarrow \textrm{Pic}(\overline X)$ is the map that sends $1$ to the class of $D$.
\item The generic fibre of $[Y/X]$ is represented by the class of $f$ in

\noindent $H^1(k(X),\mu_n)\cong k(X)^*/k(X)^{*^n}$  
\end{enumerate}
\item Let $L/k$ be a finite separable extension. Denote by $R_{L/k}(\mathbb G_m)$ the Weil restriction of scalars of $\mathbb G_m$.

We have the exact sequence of algebraic groups
\begin{eqnarray}
1\rightarrow R^1_{L/k}\mathbb G_m\rightarrow R_{L/k}\mathbb G_m\xrightarrow{N_{L/k}}
\mathbb G_m\rightarrow1
\end{eqnarray}
where $R^1_{L/k}\mathbb G_m$ is defined by the exact sequence above. Moreover we have the dual exact sequence of Gal$(\overline k/k)$-modules

\begin{eqnarray}
1\rightarrow \mathbb Z\rightarrow \mathbb Z[\mathfrak g/\mathfrak h]\rightarrow \widehat{R^1} \rightarrow1
\end{eqnarray}
\[\te{where} \quad \mathfrak g=\te{Gal}(\overline k/k), \quad \mathfrak h=\te{Gal}(\overline k/L), \quad R^1=R^1_{L/k}\mathbb G_m .
\]
Let $f\in k(X)^*$ be a function such that $\textrm{div}(f)=N_{L/k}(D)$, for some $D$ in $\textrm{Div}(X_L)$. Then there is a torsor $Y/X$ under $R^1$ such that
\begin{enumerate}
 \item type$[Y/X]:\widehat {R^1}\rightarrow \textrm{Pic}(\overline X)$ is the $\mathfrak g$-module homomorphism
 that sends $\overline {1e}$ to the class of $D$. Here $1e$ denotes the element of $\mathbb Z[\mathfrak g/\mathfrak h]$ where $e$ is the unit element of $\mathfrak g/\mathfrak h$, and
 $\overline {1e}$ its image in $\widehat {R^1}$.
\item The generic fibre of $[Y/X]$ is represented by the class of $f$ in $H^1(k(X),R^1)\cong k(X)^*/N(L(X)^*)$  
\end{enumerate}
\end{enumerate}

\end{theorem}

\begin{proof}

 This follows easily from \cite[example $2.4.2$]{CTladescente}.
\end{proof}

\section{Brauer group over $\mathbb C$}

All diagonal quartic surfaces are isomorphic over $\mathbb C$. Let $S/\mathbb C$ be the surface $x^4-y^4=z^4-w^4$. 
The aim of this section is to prove the following:

\noindent {\bf Theorem. }
{\it Let
\[t  =  \frac{x^2-y^2}{z^2-w^2} \  , \   u =  \frac{x-y}{z-w} \  , \  v =\frac{(t^3-u^2)(z+w)}{x-y}
\]
\[G   =    u^2-t^3 \ , \ F   =    v-i(t^2-1)u
\]
where $i=\sqrt{-1}$.
The following are the three non-zero elements of $\Br(S)[2]$:
\[(FG,t+1)_2+(F,t+i)_2 \quad , \quad (t,F^2G)_4+(FG,t+1)_2\quad , \quad  (t,F^2G)_4+(F,t+i)_2.
\] 
}
\qed

\noindent {\bf Remark.} This is theorem \ref{bigthm} at the end of this section. For the fact that there are exactly three non-zero elements in $\Br(S)[2]$ see the proof of theorem \ref{bigthm}. Moreover by invoking the smooth and proper base change theorems for \'etale cohomology the theorem holds as it stands for any algebraically closed field of characteristic different
from $2$. We will not use this fact in the sequel.

Let $\phi$ be the morphism $S\rightarrow\mathbb{P}^1$ given by $t=(x^2-y^2)/(z^2-w^2)$, where $t$ is the parameter of $\mathbb{P}^1$.
This morphism is well-known and can be found in eg \cite{Shaf}. It is an elliptic fibration with $6$ degenerate fibres, each of type $I_4$, consisting of four straight lines arranged in a skew quadrilateral. The component of each fibre has multiplicity $1$.

Let $C'$ denote the generic fibre of $\phi$. It is a smooth curve of genus $1$ over $\mathbb{C}(t)$, given as the intersection of two quadrics in $\mathbb{P}^3$.
\[C':x^2-y^2-t(z^2-w^2)=t(x^2+y^2)-(z^2+w^2)=0\]
The curve $C'$ has a rational divisor of degree $2$ which allows us to 
represent it as a double covering of $\mathbb{P}^1$, ramified at $4$ points. In more detail let $C$ be given by the equation
\[C:v^2=(u^2-t^3)(tu^2-1)\]
 It is easily checked that the map $u=f_1, \quad  v=\frac{(t^3-f_1^2)f_2}{t}$, where 
$f_1=\frac{x-y}{z-w}$ and $f_2=\frac{x+y}{z-w}$ realises $C'$ as the smooth projective model of $C$. In the sequel we do not distinguish $C$
 from $C'$.

For the rest of this section we fix the notation $K=\mathbb C(t)$.

\noindent The curve $C$ is in a standard form to equip it with the structure of a $2$-covering of its Jacobian. We have:
\begin{lemma}\label{lklk}
Let E be the elliptic curve  
\[E:y^2=x(x+1)(x+c^2)\]
 where $c=\frac{t^2-1}{t^2+1}$. 
Let $C\xrightarrow{\tau}E$ be given by
\[x=\frac{u^2}{v^2}(t^2-1)^2 \ , \quad y=\frac{t(t^2-1)^2}{(t^2+1)}\frac{u(u^4-t^2)}{v^3}. \]
Identifying $E$ with $\textrm{Jac}(C)$,

$\tau$ is the map that sends 
\[P\mapsto2P-\infty_1-\infty_2\in \te{Pic}^0(C)\]
 where
$\infty_1=(1:-1:\sqrt{t}:\sqrt{t})$ and $\infty_2=(1:-1:-\sqrt{t}:-\sqrt{t})$. 

$(C,\tau)$ is a $2$-covering of $E$.

\end{lemma}

\begin{proof}
Straightforward calculations following well-known formulae (see eg. \cite{Perlis}).

\end{proof}
\begin{remark} We have that $E(K)=E[4]$. The points $(c,c^2+c)$ and $(d-1,id(d-1))$, where $d=\frac{2t}{t^2+1}$, are two generators of this group (see \cite[pg. 586]{Shaf}, or \cite{SD}).
 \end{remark}

We will now study this $2$-covering more closely. We fix the following notation
\[e_1=(0,0) \quad , \quad e_2=(-c^2,0) \in E[2]\]
and we identify $E[2]$ with $\mu_2^2$ throughout via the isomorphism
\begin{eqnarray}\label{ttof}
 \mu_2^2&\xrightarrow{\beta}&E[2]\\
\beta((-1,1))=e_2 &\quad , \quad &\beta((1,-1))=e_1. \nonumber
\end{eqnarray}
Note that in the sequel we will freely identify $\mathbb{Z}/2\Z$ with $\mu_2$. Using these identifications we identify the groups $H^1(K,E[2]), H^1(K,(\mathbb{Z}/2\Z)^2),H^1(K,\mu_2^2)$. Moreover we identify the latter group with $(K^*/K^{*^2})^2$ via the map coming from the Kummer sequence. We will use all these identifications throughout without explicitly mentioning them every time (everything should be clear from the context). For example, in the next lemma  $(1,t)$ is a priori an element of $(K^*/K^{*^2})^2$ but we consider it as an element of $H^1(K,E[2])$ via the identifications above.

\begin{lemma}\label{2c}
 The class of the $2$-covering $(C,\tau)$ in $H^1(K,E[2])$ is $(1,t)$.
\end{lemma}
\begin{proof}

Let $M=K(\sqrt{t})$ and let $g$ be the generator of $\textrm{Gal}(M/K)$. As a $2$-covering $(C,\tau)$ is trivialised over $M$ because the fibre at the identity contains a point defined over $M$. Therefore $(C,\tau)$ is the image under the inflation map of a cocycle class, say $a$, from
\[
 H^1(\textrm{Gal}(M/K),E[2]).
\]
Let us find a cocycle in this class. Fix an isomorphism $C\xrightarrow{l}E$
 of $2$-coverings, defined over $M$. This choice fixes a cocycle, call it $\sigma$,  in the class $a$. In more detail, denote 
by $\sigma_g\in E[2]$ the image of $g$ under $\sigma$. Then translation by $\sigma_g$ on $E$ is 
determined by the property that
$$
\sigma_g\cdot {}^gT=l({}^gl^{-1}(T)), \quad \forall \quad T\in E(M).
$$

\noindent Choose  $P\in C(M)$  such that $\tau(P)\in E[2]$ (this is clearly possible eg.  $P=\infty_2$). We have that $l(P)\in E[4]$ and so
\[
l(^gP)=\sigma_g\cdot {}^gl(P)=\sigma_g \cdot l(P)=l(\sigma_g \cdot P)
 \] 
The first equality holds from the explicit description of $\sigma_g$ (or if you prefer because  $C$ can be obtained by Galois descent on $E$- twisting $E$ by $\sigma$). The second equality holds because $l(P)$ is defined over $K$. The third
equality holds  because $l$ is an isomorphism of $2$-coverings, and so a fortiori it is an isomorphism of $M$-torsors under $E$ as well. We have thus shown that 
\begin{eqnarray}\label{bfdb}
^gP=\sigma_g\cdot P
\end{eqnarray}
Taking $P=\infty_2$, we see that $\sigma_g=\infty_1-\infty_2\in \textrm{Jac}(C)$. By the definition of $\tau$ the latter is equal to $\tau(\infty_1)$, which by direct computation is equal to $e_1$. The image of this cocycle under the inflation map is precisely $(1,t)$. 

\end{proof}

We denote by $\lambda_2$ the natural injection $E[2]\rightarrow \textrm{Pic}(\overline C)$.

\begin{proposition}\label{kernel}
Let $J$ be the kernel of the natural map 
\[H^1(K,E[2]) \xrightarrow{\lambda_{2_*}} H^1(K,\te{Pic}(\overline{C}))\]
$J$ is isomorphic to $(\mathbb{Z}/2\Z)^3$. If we identify $E[2]$ with $\mu_2^2$ via
$\beta$, the elements 
\[(t^2+1,1),\, (t^2-1,t^2-1),\, (1,t) \in H^1(K,\mu_2^2)\]
are generators of $J$.
\end{proposition}

\begin{proof}
We have the following well-known exact sequence:
\[0\rightarrow E(K)/2E(K)\xrightarrow{\partial} H^1(K,E[2]) \xrightarrow{f} H^1(K,E)[2]\rightarrow
0
\]
and the following description for $\partial$ with respect to the identification
via $\beta$ (see eg. \cite[chapter X]{Silverman}):.
\[P \mapsto ( x(P),x(P)+c^2 )\in H^1(K,\mu_2^2)   \]
By considering two generators of $E(K)=E[4]$ (cf. remark below lemma \ref{lklk}) we see that 
\[\textrm{Ker}(f)=\langle  (t^2+1,t), \, \,(t^4-1,t^2-1)\rangle \]
Note that $(1,t)\notin \textrm{ker}(f)$ and hence $C$ has exact order $2$ in $H^1(K,E)$, in other words $C$ has no section.

Let us now calculate the kernel of the natural map $H^1(K,E) \xrightarrow{\rho}H^1(K,\textrm{Pic}(\overline{C}))$.

\noindent The short exact sequence of $\textrm{Gal}(\overline{K}/K)$-modules

\[0\rightarrow E(\overline{K}) \rightarrow \textrm{Pic}(\overline{C}) \rightarrow \mathbb{Z}
\rightarrow 0
\]
gives us the following exact sequence
%

\begin{align}\label{cg1}
\mathbb{Z}&   \rightarrow   &       H^1(K,E) &
      \xrightarrow{\rho}      &   H^1(K,\textrm{Pic}(\overline{C}))   &\rightarrow&1  \notag  \\
1     &  \mapsto    &  [C]             &           &        && 
\end{align}

\noindent In particular $\textrm{ker}(\rho)=[C]$.

Since $[C]=(1,t)\in H^1(K,E[2])$ and $\lambda_{2_*} = \rho  \circ f $ the lemma is proved.

\end{proof}

\begin{theorem}\label{the types}

Let 
\[G=u^2-t^3, \quad F=v-i(t^2-1)u \, \in K(C)\]
There exist torsors, $T_G/C$ and $T_F/C$ under $\mu_2$, such that
\begin{enumerate}
 \item Their generic fibres are represented
in $H^1(K(C),\mu_2)$
by the classes of $G$ and $F$ respectively
\item {\rm type}$(T_G)=(0,0)$ and {\rm type}$(T_F)=(-c^2,0)$. 
Here the type of a torsor under $\mu_2$ is identified 
with the image of $1\in\mathbb Z/2\Z$
under the type map $\mathbb{Z}/2\Z\to \textrm{Pic}^0(\overline{C})=E(\overline{K})$.
\end{enumerate}

\end{theorem}

\begin{proof}
The proof is an application of theorem \ref{cyclic constr}. 
We retain the notation of lemma \ref{lklk}.
Straightforward
computations yield:
\[\te{div}(G)=2(P_1+{}^gP_1-\infty_1-\infty_2), \quad
\te{ div}(F)=2(P_2+{}^gP_2-\infty_1-\infty_2),\]
 \begin{eqnarray}\label{bbeq}
 \tau(P_1)=0,  \quad   \tau(P_2)=(-1,0)=e_1+e_2, \quad \tau(\infty_1)=(0,0)=e_1
 \end{eqnarray}
 where $P_1=(t\sqrt{t},0)$, $P_2=(\sqrt{t},i(t^2-1)\sqrt{t})$, and $g$ is the non-trivial element of Gal$(K(\sqrt{t})/K)$.

By part $1$ of theorem \ref{cyclic constr} it suffices to show that the following equalities hold in Pic$(\overline C)$.
\begin{eqnarray}\label{bbform}
P_1+{}^gP_1-\infty_1-\infty_2=e_1 \ , \  P_2+{}^gP_2-\infty_1-\infty_2=e_2.
\end{eqnarray}

We remind the reader that we consider the $e_i$ as degree $0$ divisor classes in $\textrm{Pic}(\overline{C})$ in these 
equalities. Now by the proof of lemma \ref{2c}, we have that ${}^gP_1=\sigma_g\cdot P_1=e_1+P_1$ in $\textrm{Pic}(\overline{C})$.
Hence the following equalities hold in $\textrm{Pic}(\overline{C})$

\begin{eqnarray}\label{argumu}
P_1+{}^gP_1-\infty_1-\infty_2&=&P_1+\sigma_g\cdot P_1-\infty_1-\infty_2=\\
2P_1+\sigma_g-\infty_1-\infty_2&=&\tau(P_1)+\sigma_g=e_1.\nonumber
\end{eqnarray}

Similarly we establish that  $P_2+{}^gP_2-\infty_1-\infty_2=e_2$ in $\textrm{Pic}(\overline{C})$.

\end{proof}

Let  $i_1$ be the injection
of $\mu_2$ to the first factor of $\mu_2^2$, and $i_2$ be the injection
of $\mu_2$ to the second factor of $\mu_2^2$. By an abuse
of notation we denote by the same letters the induced maps on cohomology
groups. Let $T/C$ be a torsor under $\mu_2^2$ corresponding to
$i_1([T_G/C])+i_2([T_F/C])$.

Via $\beta$ (see (\ref{ttof})), $T/C$ becomes a torsor under $E[2]$. 
We chose $\beta$ so that the following will hold.
\begin{lemma}\label{lambda2}
Let $w$ denote the isomorphism $E[2]\rightarrow \widehat{E[2]}$ coming from the Weil pairing.
We have that {\rm type}$(T/C)=\lambda_2\circ w^{-1}$.

\end{lemma} 
\begin{proof}
 We know {\rm type}$(T_F/C)$ and {\rm type}$(T_G/C)$ from theorem \ref{the types}.
 The result follows from going through the definitions and the identifications.
\end{proof}

\begin{proposition}\label{asqw}
The following diagram is commutative

\centerline{
\begin{xy}
(40,30)*+{H^1(C,\widehat{E[2]})}="ka"; 
(40,45)*+{H^1(K,\widehat{E[2]})}="km";
(40,60)*+{H^1(K,E[2])}="pa"; 
(90,30)*+{\Br_1(C)}="kd";
(90,60)*+{H^1(K,\textrm{Pic}(\overline{C}))}="pd"; 
{\ar@{->}^{\cup [T/C]} "ka";"kd"};
{\ar@{->}^{\lambda_{2_*}} "pa";"pd"};
{\ar@{->}^{w_*} "pa";"km"};
{\ar@{->}^{s^*} "km";"ka"};
{\ar@{->}_{\psi} "kd";"pd"};
\end{xy}}
For $p,g\in K^*$ we have 
\[
 (s^*\circ w_*)((p,g)) \cup [T/C]    =(F,p)_2+(G,g)_2
\]

\end{proposition}
\begin{proof}
Using lemma \ref{lambda2} the diagram is commutative by theorem \ref{type theorem}.
For the last statement: $\Br(C)\hookrightarrow \Br(K(C))$ and by theorem \ref{the types} the generic fibre 
of $T/C$ is represented by $(G,F)$ in $H^1(K(C),E[2])$. Take $(p,g) \in H^1(K,E[2])$ and now work on the $K(C)$-level. Since all the
the maps involved are functorial and the cup product is bilinear, it is not hard to see by diagram chasing that 
 the cup product of $(p,g)$ with the generic fibre of $T/C$ is precisely $(F,p)_2+(G,g)_2$. 
\end{proof}

\begin{proposition}\label{AA}
 Let $\mathcal A  = (FG,t+1)_2+(F,t+i)_2$. Then $\mathcal A$ belongs to $\Br(S)$ and is not zero.
\end{proposition}
\begin{proof}
First let us show that  $\mathcal A$ is not zero. By proposition \ref{asqw}
\[\psi(\mathcal A)=\lambda_{2_*}(((t+1)(t+i),t+1))
\]
The latter is not zero by proposition \ref{kernel}. Hence $\mathcal A$ is not zero either.

 To prove that $\mathcal A$ belongs 
to $\Br(S)$ we use Grothendieck's purity theorem for the Brauer group (see \cite[Theorem 1.3.2]{CT}). We have to show that $\mathcal A$ has trivial residue at all prime divisors of $S$. Since $\mathcal A$ is given explicitly, this can be established by 
straightforward calculations (cf. remark below).

\end{proof}
\noindent {\bf Remark.}
The tables of the appendix provide information for the calculations involved in the explicit use of
Grothendieck's purity theorem.
Here we give one example, by showing that $\mathcal A$ has trivial residue at $l_5$ (notation as in appendix). 
For more examples of calculating residues see the calculations after theorem \ref{bigthm}, and proposition \ref{neothm}. 

 Let $\mathcal O_{S,l_5}$ be the discrete valuation ring of $l_5$ in $S$. Denote by $L$, $J$, $\textrm{val}_{l_5}$ its field of fractions, residue field and corresponding valuation on $L$ respectively.
We have the commutative diagram:

\centerline{
\begin{xy}
(40,30)*+{H^2(L,\mu_2)}="ka"; 
(40,50)*+{\te{Br}(L)}="pa"; 
(70,30)*+{H^1(J,\mu_2)}="kd";
(70,50)*+{H^1(J,\mathbb Q/\mathbb Z)}="pd"; 
{\ar@{->}^{\partial_{l_5}} "pa";"pd"};
{\ar@{->}^{\partial_{l_5}^{'}} "ka";"kd"};
{\ar@{->}^{} "ka";"pa"};
{\ar@{->}_{} "kd";"pd"};
\end{xy}}

\noindent where the vertical maps are the natural injections (remember that we identified $\mu_2$ with $\Z/2\Z$), $\partial_{l_5}$ is the residue map, and $\partial_{l_5}^{'}$ is the residue map corresponding to $\mu_2$. For the latter map and its basic properties we will use, see \cite[99.D]{Merk} or \cite[\S 7]{SerreInv}.

\noindent We have that $\textrm{val}_{l_5}(FG)=2$ and $\textrm{val}_{l_5}(t+1)=0$. These imply that  $$\partial_{l_5}^{'}((FG,t+1)_2)=1$$ and so a fortiori $\partial_{l_5}((FG,t+1)_2)=1$.
\noindent For the other quaternion algebra we have that $\textrm{val}_{l_5}(F)=1$ and $\textrm{val}_{l_5}(t+i)=0$. Hence $$\partial_{l_5}^{'}((F,t+i)_2)=\ov{t+i}$$
 where $\overline{\phantom{A}}$ denotes the composite map 
$$O_{S,l_5}^*\to J^* \to H^1(J,\mu_2)$$
\noindent Since $\ov{t+i}=\ov {1+i}$, and $1+i$ is a square in $J$ (as $\mathbb C \subset J$) we deduce that $\partial_{l_5}((F,t+i)_2)=\partial_{l_5}^{'}((F,t+i)_2)=1$. Therefore $\partial_{l_5}(\mathcal A)=1$, ie. $\mathcal A$ has trivial residue at $l_5$.

\noindent {\bf Remark.}
We already know that $\mathcal A\in\te{Br}(C)$, since $$\mathcal A=(s^*\circ w_*)(((t+1)(t+i),t+1))\cup [T/C].$$
Hence we need only check that $\mathcal A$ has trivial residue at prime divisors of $S$ which are vertical with respect to the fibration $\phi$.

Finally let us note  that possible ambiguities concerning the sign of the residue maps do not concern us, as we are only interested in the corresponding kernels.

\noindent {\bf Remark.}
 Let us say what we are actually doing here. Let
\[R= \psi^{-1}(\lambda_{2_*}(H^1(K,E[2]))).
 \]

\noindent Modulo constant algebras, $R$ corresponds to the information about the Brauer group of $C$ 
we get by considering all unramified double coverings of $C$ (since $C$-torsors under $\mu_2$ of the same type differ by a cocycle coming from $H^1(K,\mu_2)$).
We have exhibited elements, 
namely $(F,p)_2+(G,g)_2$, that span $R$ modulo constant algebras (actually in our case Tsen's theorem tells us all constant algebras
are trivial). Now $R$ is a subgroup of index $2$ of $\Br(C)[2]$. This is because $C$ has no section, and what is missing is the image of a cocycle class $\xi$ such that $2\xi =[C]$ in $H^1(K,E)$. $\Br(S)[2]$ contains three non-zero elements. It could happen that all three elements lie in $R$. However this is not the case- only $\mathcal A$ lies in $R$ (cf. theorem \ref{bigthm}).

\begin{lemma}\label{thediv}
Let $L=K(\sqrt[4]t)$. We use the notation of theorem \ref{the types} and its proof.
Let 
\[D'=P_2+{}^gP_2+P_1-3\infty_1\in \textrm{Div}(C_L)
\]
Let $T$ be its class in $\textrm{Pic}(\overline C )$.
Then $2T=(0,0)$ and $\textrm{div}(F^2G)=N_{L/K}(D')$.
\end{lemma}
\begin{proof}

The fact that $\textrm{div}(F^2G)=N_{L/K}(D')$ is obvious. To show that $2T=(0,0)$ we use the same argument 
we used in showing the equalities (\ref{argumu}).

\end{proof}
Note that we have implicitly chosen a primitive $4$-th root of unity in $K$, namely $i$. This has the effect of fixing an isomorphism $\mu_4\cong\mathbb Z/4$, namely sending $i$ to $1$. It also chooses a generator of Gal$(K(\sqrt[4]\alpha)/K)$, $\alpha\in K^*$. It 
is the automorphism given by 
$\sqrt[4]\alpha\mapsto i\sqrt[4]\alpha$. These facts are implicit in the following theorem.
\begin{theorem}\label{coc}
Let
\[\mathcal Z=(t,F^2G)_4\in\Br(K(C))\]
Then
\begin{enumerate}

\item $\mathcal Z\in \Br(C)$

\item Let $L=K(\sqrt[4]t)$ and $H=\textrm{Gal}(L/K)$.
 Let $\sigma_1: H\rightarrow \textrm{Pic}^0(C_L)$ be the cocycle that maps the chosen generator of $H$ to $T$. This cocycle naturally induces
a cocycle $\sigma: \textrm{Gal}(\overline K/K)\rightarrow \textrm{Pic}(\overline C) $. Note that the class of $\sigma$ is the image of the class of $\sigma_1$ under the composite map:
\[H^1(H,\textrm{Pic}^0(C_L))\rightarrow H^1(H,\textrm{Pic}(\overline C)^{\textrm{Gal}(\overline
 K/L)}))\rightarrow H^1(K,\textrm{Pic}(\overline C))
\]
where the map on the right is the inflation map. Then $\psi(\mathcal Z)$ is the class of $\sigma$.

\end{enumerate}
\end{theorem}

\begin{proof}

The proof is basically an application of theorem \ref{cyclic constr} and theorem \ref{type theorem}.

 Lemma \ref{thediv} tells us that if we take
\[L=L\, , k=K\, , X=C \, , f=F^2G \, , D=D'
\]
the conditions of part (2) of theorem \ref{cyclic constr} are satisfied. Let $Y/C$ be the torsor from theorem \ref{cyclic constr}. Now we apply theorem \ref{type theorem}. The rest of the proof is calculations with cocycles.
 
We have isomorphisms
\[H^1(H,\widehat{R^1})\cong H^1(K,\widehat{R^1})\cong H^1(H,\mathbb Q/\mathbb Z).
\]
 The first isomorphism is just the inflation map (since $I=\textrm{Gal}(\overline
 K/L)$ acts trivially on $\widehat{R^1}$ and $H^1(I,\widehat{R^1})=0$, the inflation map is indeed an isomorphism). The composite isomorphism is induced by the $\textrm{Gal}(\overline K/K)$-module (and hence $H$-module) homomorphism that maps $\overline{1e}$ to $1/4$. We identify these three cohomology groups via these isomorphisms. 

Let $\sigma_2$ be the cocyle which maps the chosen generator of $H$ to $1/4$. Clearly $\sigma_2$
 is mapped to $\sigma$ under type$(Y/C)_*$. The proof 
will be concluded once we show that on 
the $K(C)$-level the cup product of $\sigma_2$ with the class of the generic fibre of $Y/C$ is $\mathcal Z$. This follows from the commutativity of

\centerline{
\begin{xy}
(0,90)*+{R^1}="p1"; 
(10,90)*+{\times}="p2"; 
(20,90)*+{\widehat{R^1}}="p3";
(40,90)*+{\mathbb G_m}="p4";  
(0,70)*+{\mu_4}="k1"; 
(10,70)*+{\times}="k2"; 
(20,70)*+{\mathbb Z/4\mathbb Z}="k3";
(40,70)*+{\mathbb G_m}="k4";  
{\ar@{->} "p3";"p4"};
{\ar@{->} "k3";"k4"};
{\ar@{->}^{l} "k1";"p1"};
{\ar@{->}^{\widehat l} "p3";"k3"};
{\ar@{->}^{id} "p4";"k4"};
\end{xy}}

\noindent where $\mu_4\xrightarrow{l} R^1$ is the natural injection.

\end{proof}
\begin{lemma}\label{big}

 Let
\[
 R=\psi^{-1}((\lambda_2)_*(H^1(K,E[2])).
\]

 Then
\[\Br(C)[2]=R\sqcup (R+\mathcal Z)
\]
where $\sqcup$ stands for disjoint union.

\end{lemma}

\begin{proof}

We have that (cf. proof of proposition \ref{kernel})

\begin{eqnarray}\label{ddd}
H^1(K,E)/[C]\cong H^1(K,\te{Pic}(\overline{C})).
\end{eqnarray}
Let $\sigma$ be the cocycle of theorem \ref{coc} considered as taking values in $E(\overline K)$. Let $\xi$ be its class in $H^1(K,E)$, and $\overline \xi$ its class in $H^1(K,\te{Pic}(\overline C))$. The cocycles are explicitly given and it is easy to verify that $2\xi=[C]$ - remember that we know $[C]$ from lemma \ref{2c}, and in particular its description as an element of $H^1(K(\sqrt[4]t)/K,E[2])$.
Moreover $\psi(\mathcal Z)=\overline \xi$ by theorem \ref{coc}.

We note $3$ things we already know: (i) the map $\psi$ is injective (by its definition from the Hochschild-Serre spectral sequence, and the fact that $\Br(K)=0$ by Tsen's theorem), (ii) $[C]$ has order $2$ in $H^1(K,E)$ (cf. proof of proposition \ref{kernel}), and (iii) elements of $H^1(K,E)$ that are killed by $2$ come from elements of $H^1(K,E[2])$ .

The lemma follows easily from the observations above:

\noindent First we show that $\mathcal Z$ is killed by $2$. This is clear since $\psi$ is injective and

\[\psi(2\mathcal Z)=2\psi(\mathcal Z)=2\overline \xi=\overline{[C]}=0 .
\]

\noindent Now suppose that the union in the lemma is not disjoint. Then $\overline \xi =\psi(\mathcal Z)=\lambda_{2_*}(l)$ for some $l\in H^1(K,E[2])$. Hence $\xi-l$ belongs to the subgroup of $H^1(K,E)$ generated by $[C]$. In particular it is killed by $2$ and 
so $0=2(\xi-l)=2\xi=[C]$. This is a contradiction since $[C]$ is not trivial.

\noindent Finally, if $l\in \te{Br}(C)[2]$ then $\psi(l)$ can be lifted to an element $d\in H^1(K,E)$ such that $2d=0$ or $2d=[C]$. 
Since elements of $H^1(K,E)$ that are killed by $2$ come from elements of $H^1(K,E[2])$ we have that $l\in R$ in the first case and that $l\in R+\mathcal Z$ in the second case.

\end{proof}

\begin{theorem}\label{bigthm}
Let
\begin{eqnarray}
\mathcal D=\mathcal Z +(t+1,FG)_2 
\end{eqnarray}
Then $\mathcal A$, $\mathcal D$ and $\mathcal A+\mathcal D$ are the three non-zero elements of $\Br(S)[2]$.

\end{theorem}
\begin{proof}
 We know that 
\[\te{Br}(S)\cong(\mathbb Q/\mathbb Z)^{\beta_2-\rho}\]
Since $\beta_2=22$ for any $K3$ surface and $\rho=20$ for a diagonal quartic surface we have that $\te{Br}(S)[2]\cong(\mathbb Z/2)^2$. Therefore it suffices to show that $\mathcal D$ and $\mathcal A$ are two distinct, non-zero elements of $\Br(S)[2]$.

We have that $0\neq\mathcal A\in \te{Br}(S)[2]$ by proposition \ref{AA}.
Let $R$ be as in lemma \ref{big}. It is clear that $\mathcal A\in R$ and $\mathcal D\in R+\mathcal Z$ (cf. proposition \ref{asqw}, $\psi((FG,t+1))=\lambda_{2_*}((t+1,t+1)))$.
This implies (using lemma \ref{big}) that $\mathcal D$ is not equal to $\mathcal A$ and has exact order $2$.
What remains is to show that $\mathcal D\in \te{Br}(S)$. Since $\mathcal D$ is given explicitly we can use Grothendieck's purity theorem for the Brauer group to establish this by straightforward computations (cf. remark after proposition \ref{AA} and the example below).


\end{proof}

\noindent {\bf Example.} Here we show that $\mathcal D$ has trivial residue at $l_1$ and $l_{21}$ (notation as in appendix).

We use the notation of the remark after proposition \ref{AA}. In order to handle the cyclic algebra $\mathcal Z$ we introduce the following: $\partial_{l_j}^{''}$ will denote the residue map corresponding to $\mu_4$ (remember we have fixed an isomorphism $\mu_4\cong\mathbb Z/4$, by sending $i$ to $1$), and $\widetilde{\phantom{A}}$ will denote the composite map 
$$O_{S,l_m}^*\to J_m^* \to H^1(J_m,\mu_4).$$
Note that we have a similar commutative diagram as in the remark after proposition \ref{AA}, corresponding to $\partial_{l_m}^{''}$.
We remind the reader that the properties of the residue maps involved in the calculations we will do can be found in \cite[99.D]{Merk} or \cite[\S 7]{SerreInv}.

For $l_1$: We have that $\textrm{val}_{l_1}(t+1)=0$, and $\ov{t+1}=\ov {1}$. Therefore  $$\partial_{l_1}^{'}((t+1,FG)_2)=1$$
and so a fortiori $\partial_{l_1}((FG,t+1)_2)=1$.

\noindent For $\mathcal Z$ we have:

\[\begin{array}{lll}

\partial_{l_1}^{''}(\mathcal Z)=\partial_{l_1}^{''}((t,t^3s_1s_2)_4) &=&(\partial_{l_1}^{''}((t,t)_4))^3\cdot \partial_{l_1}^{''}((t,s_1s_2)_4)\\
&=&(\widetilde{-1})^3\cdot \widetilde{s_1s_2}\\
&=&\widetilde{\frac{(i-1)^2i}{y^22y^2}}=1

\end{array}\]
Hence $\partial_{l_1}(\mathcal Z)=1$. Therefore $\partial_{l_1}(\mathcal D)=1$ and we showed that $\mathcal D$ has trivial residue at $l_1$.

For $l_{21}$: First we compute $\partial_{l_{21}}^{''}((t,G)_4)$.

\[\begin{array}{lll}

\partial_{l_{21}}^{''}((t,G)_4)=\partial_{l_{21}}^{''}((t,t^3s_1)_4) &=&(\partial_{l_{21}}^{''}((t,t)_4))^3\cdot \partial_{l_1}^{''}((t,s_1)_4)\\
&=&(\widetilde{-1})^3\cdot (\widetilde{s_1})^{-1}\\
&=&(\widetilde{-1})^4=1

\end{array}\]
Let $k$ be any field, $\, a,b\in k^*$, and suppose that $\mu_4 \in k$. Then we clearly have that 
$$  (a,b^2)_4=(a,b)_2\in \Br(k) \quad   \text{(*)} $$
Hence we deduce from the computation above that $\partial_{l_{21}}^{'}((t,G)_2)=1$. This justifies the second equality in the following.
\[\begin{array}{lll}

\partial_{l_{21}}^{'}((t+1,G)_2)&=&\partial_{l_{21}}^{'}((t,G)_2)\cdot \partial_{l_{21}}^{'}((1+\frac{1}{t},G)_2)\\
&=&1\cdot \ov 1=1

\end{array}\]

\noindent By (*) we have that 
$\partial_{l_{21}}^{''}((t,F^2)_4)=\partial_{l_{21}}^{'}((t,F)_2)$. Therefore

$$\partial_{l_{21}}((t,F^2)_4+(t+1,F)_2)=\partial_{l_{21}}^{'}((t(t+1),F)_2)=\partial_{l_{21}}^{'}((1+\frac{1}{t},F)_2)=1$$

\noindent Adding all these up we have established that $\partial_{l_{21}}(\mathcal D)=1$.

\section{$\Br(S)$ over number fields}
In this section we let $L$ be a number field and we denote by $s_L$ the natural map $\te{Br}(S_L)\rightarrow\te{Br}(\overline S)$. 

The aim of this section is to prove the following theorem.

\begin{theorem}\label{qwqw}
Suppose that $i\in L$. Then the $2$-primary torsion subgroup of $\Br(S_L)$ is contained in $\Br_1(S_L)$ 
if and only if $\sqrt[4]2\notin L$.

\end{theorem}
A small remark on notation for this section. In this section $k$ will always be a number field that contains $i$. We have the symbols
\begin{eqnarray}
\mathcal A & = &(FG,t+1)_2+(F,t+i)_2 \nonumber \\
\mathcal D & = &(t,F^2G)_4+(FG,t+1)_2\nonumber \\
\mathcal E=\mathcal A+\mathcal D & = &(t,F^2G)_4+(F,t+i)_2\nonumber
\end{eqnarray}
The symbols $\mathcal A$, $\mathcal D$, $\mathcal E$ can and will be considered as elements of $\Br(k(S))$. We will denote
 $\overline{\mathcal A}=\mathcal A\otimes_{k(S)}\overline k(S)$, ie. the image of $\mathcal A$ under the corresponding restriction map. 
So in the notation of the previous section $\overline{\mathcal A}$ was denoted simply by $\mathcal A$. Similarly for 
$\mathcal D$ and $\mathcal E$. Note that although $\overline{\mathcal A}$ is contained in $\te{Br}(\overline S)$, it is not necessarily true that $\mathcal A$ is contained in $\te{Br}(S)$.

Theorem \ref{qwqw} follows from the following more precise result.

\begin{theorem}\label{bbg1}
Suppose that $i\in L$. Then $\overline{\mathcal A}\in \te{im}(s_L)$ if and only if $\sqrt[4]2\in
L$.

\noindent The same is true if we replace $\mathcal A$ by $\mathcal D$, or by $\mathcal E$.
\end{theorem}

We will prove the following equivalent theorem.
\begin{theorem}\label{bbg2}
Suppose that $i,\sqrt2\in L$. Then $\overline{\mathcal A}\in \te{im}(s_L)$ 
if and only if $\sqrt[4]2\in
L$. The same is true if we replace $\mathcal A$ by $\mathcal D$, or by $\mathcal E$.

\end{theorem}

\begin{proof}
Proposition \ref{neothm} and proposition \ref{part2}.
\end{proof}

Proof of equivalence:

\noindent Theorem \ref{bbg1} $\Rightarrow$ Theorem \ref{bbg2}: clear.

\noindent Theorem \ref{bbg2} $\Rightarrow$ Theorem \ref{bbg1}: If $\sqrt[4]
2\in L$, apply theorem \ref{bbg2}. If
$\sqrt[4]
2\notin L$, apply theorem \ref{bbg2} to $L(\sqrt 2)$. To conclude the proof we just need to note that $\sqrt[4]2\notin
L\Rightarrow \sqrt[4]2\notin
L(\sqrt 2)$, and that  $s_L$ factors through $s_{L(\sqrt
2)}$.

For the rest of this section we will additionally assume that $\sqrt2\in k$. Under these assumptions the following are obviously true:
Both  $C$ and $E$ are defined over $k(t)$, $E=$Jac$(C)$, $E(k(t))=E[4]$. $\mathcal A$, $\mathcal D$ and $\mathcal E$ belong to $\Br(C)$ and have order $2$. Pic$(\overline S)$ has trivial Gal$(\overline k/k)$ action.
From the last assertion and since Gal$(\overline k/k)$ is profinite and Pic$(\overline S)$ is free we have that $H^1(k,\te{Pic}(\overline S))$ is trivial. Therefore we deduce that $\Br_1(S)=\Br(k)$ from the Hochschild-Serre spectral sequence.
\begin{lemma}\label{ker1}
The following sequence is exact
\begin{eqnarray}
1\rightarrow H^1(k,E[4])\xrightarrow{m} H^1(k(t),\textrm{Pic}(\overline C))\xrightarrow{\textrm{Res}} H^1(\overline
k(t),\textrm{Pic}(\overline C)) \nonumber
\end{eqnarray}
where $m$ is defined as follows: The group $\textrm{Gal}(\overline k/k)$ is naturally isomorphic to $\textrm{Gal}(\overline k(t)/k(t))$. We denote by $m$ the composite of the inflation map followed by the map induced by the natural injection of $E[4]$ in $\textrm{Pic}(\overline C)$.

\end{lemma}
\begin{proof}
Let 
\[N=\textrm{Gal}(\overline{k(t)}/\overline
k(t)),\quad \Gamma=\textrm{Gal}(\overline{k(t)}/k(t)), \quad G=\textrm{Gal}(\overline k(t)/k(t))=\textrm{Gal}(\overline k/k) \]

 By the inflation-restriction sequence, the kernel of the restriction map is $H^1(G,\textrm{Pic}(\overline C)^N)$. We
 will show that the natural map 
\[H^1(G,E[4])\rightarrow H^1(G,\textrm{Pic}(\overline C)^N)
\]
is an isomorphism.

Note that $D=(1:1:\sqrt t:\sqrt t)+(1:1:-\sqrt t:-\sqrt t)$ is a divisor on $\overline C$, which is
fixed by $\Gamma$ and has degree $2$ (remember that we do not distinguish $ C$ from $ C'$).

\noindent The short exact sequence of $N$-modules
\begin{eqnarray}
 1\rightarrow E(\overline{k(t)})\rightarrow \te{Pic}(\overline C)\rightarrow
 \mathbb Z \rightarrow 1 \nonumber
\end{eqnarray}

\noindent gives rise to the following exact sequence of $G$-modules
\begin{eqnarray}\label{pell}
 1\rightarrow E(\overline k(t))\rightarrow \te{Pic}(\overline C)^N\rightarrow
 2\mathbb Z \rightarrow 1 
\end{eqnarray}
 Note that for a curve of genus $1$ the non-existence of a rational point implies that there is no divisor class of degree $1$ fixed by the absolute Galois group of the ground field (cf. exact sequence (\ref{cg1})). That is why we do not have $\mathbb Z$ on the right. We have $2\mathbb Z$ because of $D$.

Part of the long exact sequence associated to (\ref{pell}) is
\begin{eqnarray}\label{pellares}
\te{Pic}(\overline C)^{\Gamma}\rightarrow2\mathbb Z\rightarrow H^1(G,E(\overline k(t)))\xrightarrow{\lambda} H^1(G,\te{Pic}(\overline C)^N)\rightarrow
 H^1(G,2\mathbb Z) \nonumber
\end{eqnarray}
Because of $D$ the left hand map is surjective, and so $\lambda$ is injective.
Since $G$ acts trivially on $2\mathbb Z$ we have that $H^1(G,2\mathbb Z)=1$,
and so $\lambda$ is surjective. Therefore $\lambda$ is an isomorphism. We conclude by noting that $E(\overline k(t))=E[4]$.

\end{proof}

\begin{lemma}\label{fut}
Let $x\in \Br(S_k)$.

 If $s_k(x)=\overline{\mathcal A}$, then
\[x=\alpha + \mathcal A\quad \te{in} \  H^1(k(t),\te{Pic}(\overline C)),\, \te{for some} \ \alpha \in  H^1(k,E[2])
\]
\noindent The same is true if we replace $\mathcal A$ by $\mathcal D$, or by $\mathcal E$.

\end{lemma}
\begin{proof}
We will prove the statement for  $\mathcal A$. The statements for $\mathcal D$ and $\mathcal E$ are proved in a similar way.
 
We remind the reader some of the maps we will be implicitly using:
$$\Br(S_k)\hookrightarrow \Br(C) \hookrightarrow \Br(k(S))$$
$$\Br(C)\xrightarrow{\psi} H^1(k(t),\textrm{Pic}(\overline C))$$
$$H^1(k,E[2])\rightarrow H^1(k,E[4])\xrightarrow{m} H^1(k(t),\textrm{Pic}(\overline C)) $$
where $\psi$ is the map from the Hochschild-Serre spectral sequence, the left hand map in the third line is induced by the injection of $E[2]$ into E[4], and $m$ was defined in lemma \ref{ker1}.

Let $y=x-\mathcal A\in \Br(C)$, and $y'=\psi(y)$. By our assumptions $y'$ is mapped to zero in $H^1(\overline k(t),\te{Pic}(\overline C))$. Hence $y'\in H^1(k,E[4])$ by lemma \ref{ker1}. Elements of $ H^1(k,E[4])$ that are killed by $2$ come from elements of  $H^1(k,E[2])$. Therefore it suffices to show that $2y'=0$.
We have
\[2y=2x\in \te{Br}_1(S_k)=\te{Br}(k)\subset \te{Br}(k(t)).\]
 Since $\te{Br}(k(t))=\textrm{ker}(\psi)$, we
 deduce that $2y'=0$ in $H^1(k(t),\te{Pic}(\overline C))$.

\end{proof}

\begin{lemma}\label{fffut}
Let $x\in Br(S_k)$.

 If  $s_k(x)=\overline{\mathcal A}$, then 
\begin{eqnarray}
x=\mathcal A +(a,F)_2+(b,G)_2+l \ \te{in} \  \Br(k(S)),\, \te{some} \ a,b \in k^*, \ l \in \Br(k(t)).\nonumber
\end{eqnarray} 

 \noindent The same is true if we replace $\mathcal A$ by $\mathcal D$, or by $\mathcal E$.

\end{lemma}

\begin{proof}
Note that
\begin{enumerate}
 \item Let $\alpha=(a,b)\in H^1(k,E[2])$. The image of $\alpha$ in $H^1(k(t),\te{Pic}(\overline C))$ 
coincides with the image of $(a,F)_2+(b,G)_2$ (cf. proposition \ref{asqw}).
\item $\te{Br}(k(t))=\te{Ker}(\te{Br}(C)\rightarrow H^1(k(t),\te{Pic}(\overline C)))$.
\end{enumerate}
The result follows from lemma \ref{fut}.
\end{proof}

\begin{proposition}\label{neothm}

Suppose that $\sqrt[4]{2}\notin k$.
\begin{enumerate} 
\item There is no element $x\in \te{Br}(S_k)$
such that $x \mapsto \overline{\mathcal A}$.

\item There is no element $x\in \te{Br}(S_k)$
such that $x \mapsto \overline{\mathcal D}$.

\item There is no element $x\in \te{Br}(S_k)$
such that $x \mapsto \overline{\mathcal E}$.

\end{enumerate}
\end{proposition}

\begin{proof}

 We introduce the theoretical part of the proof. This is utilising Faddeev's exact sequence for $\Br(k(t))$ (see \cite[\S1.2]{CT}):
\begin{eqnarray}
1\rightarrow \te{Br}(k)\rightarrow \te{Br}(k(t))\rightarrow\bigoplus_{M} H^1(k( M) ,  \mathbb{Q}/\mathbb{Z}  )\rightarrow H^1(k,\mathbb
Q/\mathbb Z)\rightarrow
1
\end{eqnarray}
\noindent Here $M$ runs through the points of codimension $1$ of $\mathbb
P^1_k$, the maps $\te{Br}(k(t))\rightarrow H^1(k( M) ,  \mathbb{Q}/\mathbb{Z})$
are residue maps, and the right hand side maps are corestriction maps. 
 
This gives us a description of elements  $l\in\te{Br}(k(t))$ modulo elements of $\Br(k)$: Choose closed points $M_1, \cdots ,
M_n$ of $\mathbb P^1_k$. Take any characters $c_j\in H^1(k( M_j) ,  \mathbb{Q}/\mathbb{Z})$,
satisfying $\sum \te{cores}(c_j)=0$. We have $\partial_{M_j}(l)=c_j$ and $\partial_P(l)=1$ for any other closed point of $\mathbb P^1_k$.

In turn we get a description of the residues $\phi^*(l)\in \te{Br}(k(S))$ (using the fact that components of fibres of $\phi$ have multiplicity $1$): The residue is trivial everywhere, except possibly
at the components above the $M_i$. When $L$ is above $M_j$
then the residue at $L$ is the image of $c_j$ under the restriction map $ H^1(k( M_j) ,  \mathbb{Q}/\mathbb{Z})\rightarrow H^1(k(L) ,  \mathbb{Q}/\mathbb{Z})$.

We are now in a position to prove part $1$. Assume that $x\in \te{Br}(S_k)$ and $s_k(x)=\overline{\mathcal A}$. In particular $x$ has trivial residue at all codimension $1$ subvarieties of $S_k$.

By lemma \ref{fffut} we can write $x=x_1+l$, where
\begin{eqnarray}\label{fundform}
x_1=\mathcal A +(a,F)_2+(b,G)_2 , \quad l \in \te{Br}(k(t)) \quad a,b \in k^*
\end{eqnarray} 
We describe the residues of $l$ as above ie. with the choice of $M_j$'s and
$c_j$'s.

Let $L$ be a smooth fibre of $\phi$ lying above $M$, ie. $M\notin \{0,\infty,\pm1,\pm i\} $. Then $\partial_L(x_1)=1$ because $F$ and $G$ are invertible elements of $\mathcal O_{S,L}$. Moreover $k(M)$ is integrally closed in $k(L)$ and so $ H^1(k(M) ,\mathbb{Q}/\mathbb{Z})\rightarrow H^1(k(L) ,  \mathbb{Q}/\mathbb{Z})$ is injective. Since $x$ has trivial residue at $L$ we deduce that the $M_j$ can only belong to $\{0,\infty,\pm1,\pm i\} $.

In particular each
$c_j$ is a character of $k$ and the condition they need to satisfy becomes \begin{eqnarray}\label{conda}
\sum
c_j=0\in H^1(k,\mathbb{Q}/\mathbb{Z})
\end{eqnarray}

 The following table gives $\partial_Y(x_1)$ (cf. calculation below the theorem).

\begin{center}

\begin{tabular}{|cc|c||c||cc|c||c||cc|c|}
\hline
  &  &  $\partial_Y(x_1)$& &  &  &  $\partial_Y(x_1)$ & &  &  &  $\partial_Y(x_1)$ \\ \hline \hline
&$l_1$  &    $ia$ &&   &$l_5$ &  $(1+i)ab$  &&   &$l_{13}$ &  $(1+i)b$ \\  \hline    
$0$ &$l_2$  &    $1$  &&   $1$ &$l_6$& $1$  && $i$  &$l_{14}$ & $1$    \\  \hline  
&$l_3$  &    $ 1$   &&   &$l_7$ & $1$  &&   &$l_{15}$ & $1$ \\ \hline  
&$l_4$  &    $ ia$  &&   &$l_8$ & $ (1+i)ab$  &&   &$l_{16}$ &  $(1+i)b$ \\  \hline \hline  

&$l_{21}$  &     $ab$  &&   &$l_9$ &  $ -(1+i)ab$  &&   &$l_{17}$ & $\sqrt{2}b$   \\  \hline  
$\infty$&$l_{22}$  &    $b$   &&  $-1$ &$l_{10}$ &  $ i$   && $-i$  &$l_{18}$ &  $-\sqrt{2}(1-i)$ \\  \hline  
&$l_{23}$  &     $b$    &&   &$l_{11}$ &  $ -i$    &&   &$l_{19}$ &  $\sqrt{2}(1-i)$ \\  \hline  
&$l_{24}$  &    $ ab$  &&   &$l_{12}$ & $ (1+i)ab$   &&   &$l_{20}$ & $-\sqrt{2}b$ \\  \hline 

\end{tabular}

\end{center}
where the values are to be considered as elements of $H^1(k(Y),\mu_2)$.

Note that if $Y$ and $Y'$ are components of the same fibre then $\partial_Y(l)=\partial_{Y'}(l)$ and so $\partial_Y(x_1)=\partial_{Y'}(x_1)$.
Let $\sim$ mean up to squares in $k$ (or $k(Y)$). We have $a\sim 1$ (looking at $l_{21}$
and $l_{22}$) and $b\sim 1+i$ (looking at $l_{13}$). Now all residues are trivial except above $-i$ and $\infty$ where the residues are $\sqrt{2}(1+i)$
and $1+i$ respectively. Therefore $c_{-i}=\sqrt{2}(1+i)$,
and $c_{\infty}=1+i$. Hence $c_{-i}+c_{\infty}=\sqrt{2}\neq1\in H^1(k,\mathbb
Z/2)$ since $\sqrt[4]{2}\notin k$. This contradicts the condition given by (\ref{conda}).

The proofs of part $2$ and part $3$ are similar. In particular we have: 
\[x_2=\mathcal D +(a,F)_2+(b,G)_2
\]
\begin{center}

\begin{tabular}{|cc|c||c||cc|c||c||cc|c|}
\hline
  &  &  $\partial_Y(x_2)$& &  &  &  $\partial_Y(x_2)$ & &  &  &  $\partial_Y(x_2)$ \\ \hline \hline
&$l_1$  &    $a$ &&   &$l_5$ &  $ab$  &&   &$l_{13}$ &  $\sqrt 2 b$ \\  \hline    
$0$ &$l_2$  &    $1$  &&   $1$ &$l_6$& $1$  && $i$  &$l_{14}$ & $1$    \\  \hline  
&$l_3$  &    $ 1$   &&   &$l_7$ & $1$  &&   &$l_{15}$ & $1$ \\ \hline  
&$l_4$  &    $ a$  &&   &$l_8$ & $ab$  &&   &$l_{16}$ &  $\sqrt 2 b$ \\  \hline \hline  

&$l_{21}$  &     $ab$  &&   &$l_9$ &  $ ab$  &&   &$l_{17}$ & $2i\sqrt{2}b$   \\  \hline  
$\infty$&$l_{22}$  &    $b$   &&  $-1$ &$l_{10}$ &  $ i$   && $-i$  &$l_{18}$ &  $1$ \\  \hline  
&$l_{23}$  &     $b$    &&   &$l_{11}$ &  $i$    &&   &$l_{19}$ &  $1$ \\  \hline  
&$l_{24}$  &    $ ab$  &&   &$l_{12}$ & $ ab$   &&   &$l_{20}$ & $2i\sqrt{2}b$ \\  \hline 

\end{tabular}

\end{center}
where the values are to be considered as elements of $H^1(k(Y),\mu_2)$.

Using the same arguments as for part $1$, we get $a\sim 1$, $b\sim i$, and $b\sim\sqrt 2$. Since $i$ is a square
in $k$, $\sqrt 2$ is a square as well. This contradicts the fact that $\sqrt[4]{2}\notin k$.

\[x_3=\mathcal E +(a,F)_2+(b,G)_2
\]

\begin{center}

\begin{tabular}{|cc|c||c||cc|c||c||cc|c|}
\hline
  &  &  $\partial_Y(x_3)$& &  &  &  $\partial_Y(x_3)$ & &  &  &  $\partial_Y(x_3)$ \\ \hline \hline
&$l_1$  &    $ia$ &&   &$l_5$ &  $(1+i)ab$  &&   &$l_{13}$ &  $\sqrt 2(1+i) b$ \\  \hline    
$0$ &$l_2$  &    $1$  &&   $1$ &$l_6$& $1$  && $i$  &$l_{14}$ & $1$    \\  \hline  
&$l_3$  &    $ 1$   &&   &$l_7$ & $1$  &&   &$l_{15}$ & $1$ \\ \hline  
&$l_4$  &    $ ia$  &&   &$l_8$ & $(1+i)ab$  &&   &$l_{16}$ &  $\sqrt 2(1+i) b$ \\  \hline \hline  

&$l_{21}$  &     $ab$  &&   &$l_9$ &  $-(1+i)ab$  &&   &$l_{17}$ & $2ib$   \\  \hline  
$\infty$&$l_{22}$  &    $b$   &&  $-1$ &$l_{10}$ &  $ 1$   && $-i$  &$l_{18}$ &  $-\sqrt 2(1-i)$ \\  \hline  
&$l_{23}$  &     $b$    &&   &$l_{11}$ &  $-1$    &&   &$l_{19}$ &  $\sqrt 2(1-i)$ \\  \hline  
&$l_{24}$  &    $ ab$  &&   &$l_{12}$ & $(1+i)ab$   &&   &$l_{20}$ & $-2ib$ \\  \hline 

\end{tabular}

\end{center}
where the values are to be considered as elements of $H^1(k(Y),\mu_2)$.

Using the same arguments as for part $1$, we get $a\sim 1$, $b\sim 1+i$, and $b\sim \sqrt 2(1+i)$. Hence $\sqrt
2$ is a square in $k$, contradiction.

\end{proof}

{\bf Calculation.} 
Here we will calculate $\partial_{l_{13}}(x_2)$. The other entries in the tables above are established in a similar way. We use the notation of the example after theorem \ref{bigthm}. The properties of the residue maps involved in the calculations we will do can be found in \cite[99.D]{Merk} or \cite[\S 7]{SerreInv}.

\noindent Since $F,t,t+1,a\in \mathcal O_{S,l_{13}}^*$, we have that 
$$\partial_{l_{13}}((t,F^2)_4+(F,t+1)_2+(a,F)_2)=1.$$

\noindent Note that we have the natural injections  $$H^1(k(l_{13}),\mu_2)\xrightarrow{g} H^1(k(l_{13}),\mu_4)\to H^1(k(l_{13}),\Q/\Z)$$ and that
$g(\ov q)=\widetilde{q^2}$ (remember we have fixed an isomorphism $\mu_4\cong\mathbb Z/4$, by sending $i$ to $1$).

\noindent Since $\textrm{val}_{l_{13}}(G)=1$ we have the following equalities (taking place in $H^1(k(l_{13}),\mu_4)$):

\[\begin{array}{lll}

\partial_{l_{13}}((t,G)_4+(G,t+1)_2)&=&\partial_{l_{13}}^{''}((t,G)_4)\cdot \partial_{l_{13}}^{'}((G,t+1)_2) \\
&=&(\widetilde{t})^{-1}\cdot \ov{i+1}\\
&=&\widetilde{-i}\cdot \widetilde{(i+1)^2} \\
&=&\widetilde{2}=\overline{\sqrt{2}}

\end{array}\]

\noindent As $\textrm{val}_{l_{13}}(b)=0$ we have $\partial_{l_{13}}((b,G)_2)=\ov b$.
Adding all these up we have established that $\partial_{l_{13}}(x_2)=\ov{\sqrt {2} b}$.

\begin{definition}\label{definitionofB}
We define
\begin{eqnarray}
\mathcal B & = & \mathcal A +(1+i,(t+i)G)_2 \nonumber \\
\mathcal E_1 & = &\mathcal E +(1+i,(t+i)G)_2\nonumber
\end{eqnarray}
\end{definition}

\begin{proposition}\label{part2}
Let $F$ be a number field containing $i$ and $\sqrt[4] 2$. Then $\mathcal B$, $\mathcal D$, $\mathcal E_1$ belong to $\Br(S_F)$. Moreover, we have that $\mathcal
B\mapsto \overline{\mathcal A}$, $\mathcal
D\mapsto \overline{\mathcal D}$, $\mathcal
E_1\mapsto \overline{\mathcal E}$.
\end{proposition}

\begin{proof}
The last line is clear. What we need to show is that each element belongs to $\Br(S_F)$.
In other words that each element has trivial residue at all prime divisors of $S_F$. A glance at the tables of proposition \ref{neothm} establishes the result: 
\begin{eqnarray}
\mathcal B & \leftrightarrow & a=1 \ , \ b=1+i  \ , \   l=(1+i,t+i)_2 \nonumber \\
\mathcal D & \leftrightarrow & a=1 \ , \     b=1  \ , \     l=1 \nonumber \\
\mathcal E_1 &\leftrightarrow & a=1  \ , \    b=1+i \ , \   l=(1+i,t+i)_2 \nonumber
\end{eqnarray}
\end{proof}

\section{Varying the coefficients}

To make the presentation in this section nicer we introduce the following terminology.

\begin{definition}
\begin{enumerate}
 \item

Let $V$ be a variety over a number field $L$. Let {\bf ST}$(V,L)$ be the statement: The $2$-primary torsion subgroup of $\Br(V_L)/\Br_1(V_L)$ is
trivial.
\item Let $a_i\in \mathbb Q^*$. For $(a_0:a_1:a_2:a_3)$, let Z be the condition
\[\sqrt[4]2 \notin \mathbb Q(i,\sqrt[4]{\frac{a_1}{a_0}},\sqrt[4]{\frac{a_2}{a_0}},\sqrt[4]{\frac{a_3}{a_0}})
\]
We define
\[W:=\{(a_0:a_1:a_2:a_3)  |    (a_0:a_1:a_2:a_3)\quad \te{satisfies condition Z}\}
\]
\end{enumerate}

\end{definition}

The main result of this section is the following theorem.
\begin{theorem}\label{ccprop}
Let $X$ be the surface 
\[a_0X_0^4-a_1X_1^4=a_2X_2^4-a_3X_3^4 \ ,\quad  a_i\in\mathbb Q^*\] 

If $(a_0:a_1:a_2:a_3)\in W$
then  {\bf ST}$(X,\mathbb Q)$ is true.

\noindent Equivalently: If neither $2$ nor $-2$ belong to the subgroup of $\mathbb Q^*/\Q^{*4}$
generated by the ratios $\frac{a_i}{a_j}$ then {\bf ST}$(X,\mathbb Q)$ is true.

\noindent Equivalently: If $2$ does not belong to the subgroup of $\mathbb Q^*/\Q^{*4}$
generated by the ratios $\frac{a_i}{a_j}$ and $-4$ then {\bf ST}$(X,\mathbb Q)$ is true.
\end{theorem}
We give two applications of this theorem at the end of this section. To prove theorem \ref{ccprop} we will need the following two lemmas.
\begin{lemma}\label{qop}
\begin{enumerate}

\item Let $L$ be a number field containing $i$. Then {\bf ST}$(S,L)$ is true
if and only if $\sqrt[4] 2\notin L$.

\item Let $L\subseteq F$ be an algebraic field extension. If  {\bf ST}$(V,F)$ is true, then {\bf ST}$(V,L)$
is true.

\end{enumerate}
\end{lemma}

\begin{proof}
\begin{enumerate}

\item  Note that if {\bf ST}$(S,L)$ is not true then there is an element of exact order $2$ in $\Br(V_L)/\Br_1(V_L)$.
Therefore this part is just a restatement of theorem \ref{bbg1}.

\item  Clearly $\Br(V_L)/\Br_1(V_L)\hookrightarrow \Br(V_F)/\Br_1(V_F)$.
Moreover, under any group homomorphism the $2$-primary part maps into
the $2$-primary part. Hence the result.

\end{enumerate}
\end{proof}

\begin{lemma}\label{woul}
Let $F=\mathbb Q(i,\sqrt[4]{a_1},\cdots
,\sqrt[4]{a_n})$, $a_i\in \mathbb Q^*$. 

Then the following are equivalent
\begin{enumerate}
 \item  $\sqrt[4]2\in F$.
\item The class of $2$ or $-2$ belongs to the subgroup of $\mathbb Q^*/\Q^{*4}$
generated by the $a_i$.
\item The class of $2$  belongs to the subgroup of $\mathbb Q^*/\Q^{*4}$
generated by the $a_i$ and $-4$.
\end{enumerate}

\end{lemma}
\begin{proof}
Let $L=(F^{*4}\cap \mathbb Q(i)^*)/\mathbb Q(i)^{*4}$, and let $f$ denote the restriction map $$\mathbb Q^*/\mathbb Q^{*4}\xrightarrow{f}\mathbb Q(i)^*/\mathbb Q(i)^{*^4}.$$
Clearly $\sqrt[4]2\in F$ if and only
if the class of $2$ in $\Q(i)^*/\mathbb Q(i)^{*^4}$ belongs to $L$ if and only
if the class of $2$ in $\Q^*/\mathbb Q^{*^4}$ belongs to $f^{-1}(L)$. 

Now we note that

(i) $F/\mathbb Q(i)$ is a $4$-Kummer extension, and so $L$ is generated by the classes
of the $a_i$ by standard properties (see eg. \cite{Morandi}).

(ii) By the inflation-restriction sequence ker$f=H^1(\mathbb Q(i)/\mathbb Q, \mu_4)\cong \mathbb Z/2$.
 Since $-4=(\frac{2}{1+i})^4$ we deduce that ker$f$ is generated by the class of $-4$.  

Therefore $f^{-1}(L)$ is generated by the $a_i$ and $-4$. This establishes the equivalence of 1 and 3.

 The equivalence of 2 and 3 is almost trivial, since $(-2)(-4)\equiv \frac{1}{2} \mod \Q^{*^4}$.
\end{proof}

{\it Proof of theorem \ref{ccprop}}:
Let $L=\mathbb Q(i,\sqrt[4]{\frac{a_1}{a_0}},\sqrt[4]{\frac{a_2}{a_0}},\sqrt[4]{\frac{a_3}{a_0}})$. The surface $X$ is isomorphic to $S$ over $L$.   If $(a_0:a_1:a_2:a_3)\in W$ then $\sqrt[4]2\notin L$. Therefore both {\bf ST}$(X,L)$ and {\bf ST}$(X,\mathbb Q)$ hold by lemma \ref{qop}. The last part of the theorem is immediate from lemma \ref{woul}.

In view of theorem \ref{ccprop}, we study the set $W$ in more detail.

\begin{lemma}\label{equi}
\begin{enumerate}
\item Condition Z is equivalent to the condition (which we will also call
condition Z)
\begin{eqnarray}\label{opre}
\sqrt[4]2 \notin \mathbb Q(i,\sqrt 2,\sqrt[4]{\frac{a_1}{a_0}},\sqrt[4]{\frac{a_2}{a_0}},\sqrt[4]{\frac{a_3}{a_0}})
\end{eqnarray}

\item If condition Z holds for $(a_0:a_1:a_2:a_3)$ then it holds for any permutation
of the $a_i$. It also holds for $(la_0:la_1:la_2:la_3)$.

\item Suppose that condition Z holds for $(a_0:a_1:a_2:a_3)$. We can multiply
any $a_i$ by a fourth power, by $-1$, or by $4$ and condition Z will still hold.
\item Let $a,b,c,d$ be non-zero integers. Suppose $(a:b:c:d)\in W$. Let $p$ be an odd prime that does not divide $abcd$. Then 
\[(p^{l_1}a:p^{l_2}b:p^{l_3}c:p^{l_4}d)\in W\]
for all integers $l_1$,$l_2$,$l_3$,$l_4$.
\item Let $a,b,c,d$ be odd integers. Then $(a:b:c:d)\in W$.
\item Let $a,b\neq\pm1$ be $4$-th power free, odd integers .

 Then $(1:1:2a:2b)\notin W$ if and only if $(a,b)$
(or of course $(b,a)$) is one of the following forms:

\begin{enumerate}
\item 
\begin{eqnarray}
a&=&\pm p_1^{a_1}\cdots p_n^{a_n}\cdot q_1^2\cdots q_r^2 \quad a_i\in\{1,3\}
\nonumber\\
b&=&\pm p_1^{2}\cdots p_n^{2} \nonumber
\end{eqnarray}

\item
\begin{eqnarray}
 a&=&\pm p_1^{a_1}\cdots p_n^{a_n}  \quad a_i\in\{1,3\} \nonumber\\
b&=&\pm p_1^{2}\cdots p_n^{2}\nonumber
\end{eqnarray}

\end{enumerate}
with $p_i,q_j$ distinct odd primes.

\end{enumerate}
\end{lemma}

\begin{proof}

For part $1$ it suffices to remark that if $i\in F$ then
$\sqrt[4]2\notin F$ implies that $\sqrt[4]2\notin F(\sqrt2)$.
Part $2$ is obvious. For part $3$ note that both $-1$ and $4$ are $4$-th powers in $\mathbb Q(i,\sqrt 2)$ and so the extension in (\ref{opre}) is unchanged. Part $4$ follows easily from lemma \ref{woul}. For part $5$ note that $(1:1:1:1)\in W$ and apply part $4$ successively for every odd prime dividing $abcd$.

Let us now prove part $6$. If $(a,b)$
is one of the forms described in the lemma, then $(2a)^2(2b)^3\equiv\pm2$
in $\mathbb Q^*/\Q^{*4}$. Hence $(1:1:2a:2b)\notin W$ by lemma \ref{woul}.
Conversely suppose that $(1:1:2a:2b)\notin W$. By lemma \ref{woul} we have  that

\begin{eqnarray}\label{conf}
\pm2\equiv2^{\epsilon_1+\epsilon_2}a^{\epsilon_1}b^{\epsilon_2}, \quad \te{in}
\quad
 \mathbb Q^*/\Q^{*4},\quad \te{some}\quad\epsilon_i\in\{0,1,2,3\}
\end{eqnarray}
By looking at the power of $2$, one of the $\epsilon_i$ is even and
the other is odd. Without loss of generality (we can interchange $a$ and $b$) $\epsilon_1$ is even. We cannot
have $\epsilon_1=0$ because $b\neq\pm1$ and is $4$-th power free. Hence $\epsilon_1=2$, and $\epsilon_2\in\{1,3\}$.
Now we look at (powers of odd primes in) (\ref{conf}), and note the following:
\begin{enumerate}
\item If a prime divides $b$ but not $a$ then it appears in the right
hand side of (\ref{conf}) to a
power not divisible by $4$. This is a contradiction. Hence if a prime divides $b$ then it divides $a$ as well.

\item If a prime divides $a$ but not $b$ then it divides $a$ to an even
power.

\item If a prime $p$ divides both $a$ and $b$ then it cannot divide $a$ to an
even power (same argument as in part $1$, as there would be no contribution from
$a$). Hence $p$ divides $a$ to an odd power. Necessarily it divides
$b$ to an even power.
\end{enumerate}
By combining the above, we see that $(a,b)$
has one of the forms described in the lemma.

\end{proof}

Let us now give two applications of theorem \ref{ccprop}.

\begin{theorem}\label{gen}
Let $X$ be the surface

\begin{eqnarray}
a_0X_0^4-a_1X_1^4=a_2X_2^4-a_3X_3^4
\end{eqnarray}
with $a_i$ integers. If $2$ divides every $a_i$ to the same parity then {\bf ST}$(X,\mathbb Q)$ holds.

\end{theorem}

\begin{proof}

By theorem \ref{ccprop} it suffices to show that $(a_0:a_1:a_2:a_3)\in W$. To do this we apply lemma \ref{equi} three  times: By part $2$ we can assume that the parity is even. By part $3$ we can assume
that that the $a_i$ are odd integers. We are done by part $5$.

\end{proof}

\begin{theorem}
 Let $X$ be the surface (taken from \cite[\S 8]{SD}).
\begin{eqnarray}
X_0^4+4X_1^4=da^2X_2^4+db^2X_3^4
\end{eqnarray}
Without loss of generality $d$ is fourth-power free, and not divisible by $4$. Moreover $a,b$ are square-free, positive and coprime, with $a\geq b$.

The only case
when {\bf ST}$(X,\mathbb Q)$ might not hold is when $d=\pm2$ and $b=1$ or $2$.
\end{theorem}

\begin{proof}
If $d$ is odd then the assertion follows as a special case of theorem \ref{gen}.
Hence we can assume that $d=2d'$ with $d'$ odd. In view of theorem \ref{ccprop} it suffices to show that $(1:4:2d'a^2:2d'b^2)\notin W$ implies that $d'=\pm1$ and $b=1$ or $2$. By part $3$ of lemma \ref{equi} we can assume that $a$ and $b$ are odd and what we want to show is that $(1:1:2d'a^2:2d'b^2)\notin W$ implies that $d'=\pm1$ and $b=1$. Under our assumptions this follows immediately from part $6$ of lemma \ref{equi}.

\end{proof}

\begin{remark}
 
The surface $X$ is the general diagonal quartic surface, with the product of coefficients
 being a square and $a_1=4a_0$. If $a=b=1$, then $X$ is everywhere locally soluble if and
only if $d>0$, $d\equiv1,2,5 \ \textrm{or} \ 10 \mod 16$, and $d$ is not divisible
by any prime $p\equiv3 \mod 4$. When $d=2$, $X$ contains rational points. For these facts see \cite[\S 1 and \S 8]{SD}.
\end{remark}

\section{Example of a transcendental obstruction to weak approximation}
In this section we give an application of a transcendental element. In particular we show that one such element provides an obstruction to weak approximation on a diagonal quartic surface over a number field. This is a genuine transcendental obstruction in the sense that over this number field the algebraic Brauer group of our surface is just the constant algebras and hence they give no conditions on the adelic points.

Let us fix some notation for this section. We  denote by $M$ the Galois field of the polynomial $s^4-2$ over $\mathbb{Q}$, that is $M=\mathbb Q(i,\sqrt[4]2)$. There is a unique place in $M$ above $2$ with ramification index $8$. We will denote it by $u$.

 We remind the reader that $S_M$ denotes the variety $x^4-y^4=z^4-w^4$ over $M$ and $\mathcal B$ is 
the element of $\Br(S_M)$ defined in definition \ref{definitionofB}. If $F/M$ is a field extension and $A\in S(F)$ we denote by $\mathcal B(A)$ the element of $\Br(F)$ we get by specialising $\mathcal B$ at $A$.

The main result of this section is the following. 

\begin{theorem}\label{weakoverM}
$\mathcal{B}$ gives an obstruction to weak approximation on
$S_M$.
\end{theorem}
The proof of the theorem is given at the end of this section.

\noindent {\bf Remark.} It is easy to see that Gal$(\overline M/M)$ acts trivially on Pic$(\overline S)\cong \mathbb Z^{20}$, and so $\Br_1(S_M)=\te{Br}(M)$ (cf. comments above lemma \ref{ker1}), ie. the algebraic Brauer group of $S_M$ is just the constant algebras.

 We will have to calculate invariants of specialisations of $\mathcal{B}$. By straightforward manipulations we have the following representation for $\mathcal{B}$ (we denote the quaternion algebra $(a,b)_2$ by $[a,b]$ in this section):
 \begin{equation}\label{repofB}
\mathcal{B}=[A_1 , A_2] + [A_3 , A_4] + [1+i , A_5] 
\end{equation}

\noindent where

\[\begin{array}{ccc}

A_1   &  =  & (1+i)(z-1)(x-y)B \\
A_2   &  =  &(-1+z^2+x^2-y^2)(-1+z^2-ix^2+iy^2)  \\
A_3   &  =  & 2(z^2-1)(-yx^3+y^3x-z+z^3) \\
A_4   &  =  & (1+i)(z^2-1)(-1+z^2+x^2-y^2) \\
A_5   &  =  & (z^2-1)(x^2-y^2+iz^2-i) \\
B     &  =  & y^2x^2-z-i+z^3+y^3x+z^4-x^4-yx^3\\
      &     & +iy^4-iyx^3+iy^3x-iy^2x^2-iz+iz^3+iz^2-z^2\\ 
\end{array}
\]

\noindent For ease with the manipulations mentioned above note that 

$$
F=\frac{(i-1)(x-y)B}{(z-1)^3(z+1)^2}, \quad       G=\frac{2(x-y)^2(x^3y-xy^3-z^3+z)}{(1-z^2)^4}
$$
$$
[A_1,A_2]=[F,(t+1)(t+i)], \quad [A_3,A_4]=[G,(t+1)(1+i)] \quad  [1+i,A_5]=[1+i,t+i].
$$

 Here and for the rest of this section we use affine coordinates on the open set $w=1$.

\begin{lemma}\label{otherpoint}
Let $Q=(2,-1,2)\in S(M)$. Then $ \textrm{inv}_u( \mathcal{B}(Q))=0$.

\end{lemma}
\begin{proof}
Consider $\mathcal{B}(Q)$ as an element of $\Br(\mathbb{Q}_2(i,\sqrt[4]{2}))$. By looking at the formulae (\ref{repofB}), it is obvious 
that $\mathcal{B}(Q)$
is the image of an element of order $2$ under the restriction 
map 
\[\Br(\mathbb{Q}_2(i))\xrightarrow{\textrm{Res}} \Br(\mathbb{Q}_2(i,\sqrt[4]{2})).\]
 Therefore it is trivial (cf. proposition \ref{Serr}).
\end{proof}
{\bf Remark.}
It is easy to establish that $\mathcal{B}(Q)$ is actually trivial as an element of $\Br(M)$ (by using equations (\ref{repofB})).

Our aim is to find an $M_u$-point of $S$ such that $\mathcal B$ specialises to a non-trivial algebra at that point.
There is nothing conceptual in finding such a point. However it is not that straightforward: the equations for $\mathcal B$ are not very simple and moreover we must make calculations of the invariant map for a totally ramified extension of $\mathbb Q_2$ of degree $8$.
The rest of this section is devoted to finding such a point.

For the convenience of the reader we quote a number theoretic result.
\begin{lemma}\label{localserre}
Let $F$ be a field complete under a discrete valuation $v$. Suppose that $F$
 has characteristic zero and its residue field has characteristic $p\neq0$.
 Let $e$ be the absolute ramification index of $F$ and let $V^{(m)}=\{x\in F | val_v(x-1)\geq m\}$.
 Then for $m > \frac{e}{p-1}$, the map $x\mapsto
x^p$ is an isomorphism of $V^{(m)}$ onto $V^{(m+e)}$.

\end{lemma}

\begin{proof}
\cite[XIV, \S 4, proposition 9]{Serre}. 
\end{proof}

\begin{lemma}\label{l}
Let $F=\mathbb{Q}_2(\sqrt{2})$ and $d=1-8\sqrt{2}(3+2\sqrt{2})$. There exists $l\in F$ such that 
\begin{eqnarray}\label{eqq}
l^4&=&d \nonumber \\
l^2&\equiv&1+4\sqrt{2}+16\sqrt{2}+32\mod 2^5\sqrt{2}
\end{eqnarray}   

\end{lemma}

\begin{proof}
By lemma \ref{localserre}, $d$ is a $4$-th power in $F$. It is easy to see that we can choose $\alpha=\sqrt d$ with $\alpha\equiv1+4\sqrt{2}+16\sqrt{2}+32\mod 2^5\sqrt{2}$. By lemma \ref{localserre}, $\alpha$ is a square in $F$.
Hence we can choose $l\in F$ with
the properties described in the lemma.

\end{proof}

Now we define the point $P=(x_0,y_0,z_0)$ by:
\begin{eqnarray}\label{pointP}
x_0=\frac{1}{\sqrt[4]{2}},\quad y_0=lx_0,\quad z_0=1+\sqrt{2} \nonumber
\end{eqnarray}
where $l$ is defined in lemma \ref{l}. It is clear that $P\in S(\mathbb{Q}_2(i,\sqrt[4]{2}))$.
By using lemma \ref{localserre} and (\ref{eqq}) we can easily establish that
\begin{eqnarray}\label{uptosq}
z_0^2-1\sim1+ \sqrt{2},\quad l^2+1\sim1+2\sqrt{2}, \quad z_0^2-l^2\sim 1-\sqrt{2}
\end{eqnarray} 
\noindent where $\sim$ means up to squares in $\mathbb{Q}_2(\sqrt{2})$.

For the convenience of the reader we quote another well-known result which lists some properties of the 
restriction and corestriction homomorphisms. We will be using them in the proof of proposition \ref{trul}.

\begin{proposition}\label{Serr}

Suppose that $K$ is a field which is complete under a discrete valuation, with perfect residue
field. Let $L/K$ be a finite extension. Then the following are commutative:

\centerline{
\begin{xy}
(40,0)*+{\Br(K)}="ka";
(80,0)*+{\textrm{Hom}(\mathfrak{g}_K,\mathbb{Q}/\mathbb{Z})}="kd";  
(40,20)*+{\Br(L)}="pa";
(80,20)*+{\textrm{Hom}(\mathfrak{g}_L,\mathbb{Q}/\mathbb{Z})}="pd";
{\ar@{->}^{\partial_{\mathcal{O}_K}} "ka";"kd"};
{\ar@{->}_{\textrm{Cor}} "pa";"ka"};
{\ar@{->}^{\partial_{\mathcal{O}_L}} "pa";"pd"};
{\ar@{->}_{\textrm{Cor}} "pd";"kd"};
\end{xy}}
\bigskip

\centerline{
\begin{xy}
(40,0)*+{\Br(K)}="ka";
(80,0)*+{\textrm{Hom}(\mathfrak{g}_K,\mathbb{Q}/\mathbb{Z})}="kd";  
(40,20)*+{\Br(L)}="pa";
(80,20)*+{\textrm{Hom}(\mathfrak{g}_L,\mathbb{Q}/\mathbb{Z})}="pd";
{\ar@{->}^{\partial_{\mathcal{O}_K}} "ka";"kd"};
{\ar@{->}_{\textrm{Res}} "ka";"pa"};
{\ar@{->}^{\partial_{\mathcal{O}_L}} "pa";"pd"};
{\ar@{->}_{e\cdot \textrm{Res}} "kd";"pd"};
\end{xy}}

where $\mathfrak{g}_K$ (resp.  $\mathfrak{g}_L$) is the absolute Galois group of the residue field of $\mathcal{O}_K$ (resp. $\mathcal{O}_L$).\\

Let $L/K$ be a finite separable extension. Denote by $(a,b)_K$ (resp. $(a,b)_L$)
the symbol $(a,b)$ computed in $K$ (resp. $L$). Then
\[\textrm{Cor}((a,c)_L)=(a,Nc)_K \quad \te{for}  \  a \in K^*, \   c\in L^*
\]
\[\textrm{Res}((a,b)_K)=(a,b)_L \quad \te{for}  \  a \in K^*, \  b\in K^*
\]
where $N$ denotes the norm map $L^*\xrightarrow{N}K^*$.
\end{proposition}
\begin{proof}
 \cite[XII, \S 3, ex. 2, and XIV, \S 2, ex. 4]{Serre}. 
\end{proof}

Let us note one easy consequence
\begin{lemma}\label{tocompute}
We retain the notation of proposition \ref{Serr}. Let $L=K(\sqrt{d})$ be
a quadratic extension. Then 
\[\textrm{Cor}([\sqrt{d},a+b\sqrt{d}]_L)=[a,-d]_K+[-ab,a^2-db^2]_K \quad \te{for}  \  a, \ b \in K^*
\]
\end{lemma}
\begin{proof}
Clearly $[x,y^2-x]=1$ for $x,y\in L$ with $x(y^2-x)\neq 0$. Appplying this for $x=-ec$ and $y=c$, it is easy to see that $[e,c]=[-ec,e+c]$
 for $e,c\in L^*$, $e+c\neq0 $. Substituting $e=a$, $c=b\sqrt{d}$ in the last equation and using proposition \ref{Serr} gives the result.
\end{proof}

\begin{proposition}\label{trul}
$\textrm{inv}_u(\mathcal{B}(P))=\frac{1}{2}$

\end{proposition}
\begin{proof}
Let $u_1$ be the place of $\mathbb{Q}_2(i,\sqrt{2})$, and let $u_2$ be the place of $\mathbb{Q}_2(\sqrt{2})$.

First note that by definition
\begin{eqnarray}\label{alll}
z_0, \, x_0^2, \, y_0^2, \, x_0y_0 \in \mathbb{Q}_2(\sqrt{2})
\end{eqnarray}
As a consequence
\begin{eqnarray}\label{eqq1}
\te{inv}_u(\mathcal{B}(P))=\te{inv}_u([x_0-y_0,(-1+z_0^2+x_0^2-y_0^2)(-1+z_0^2-ix_0^2+iy_0^2)])\nonumber
\end{eqnarray}
since by looking at the formulae (\ref{repofB}) we see that everything else appearing in $\mathcal{B}(P)$
is the image of an element of order $2$ under the restriction 
map 
\[\te{Br}(\mathbb{Q}_2(i,\sqrt{2}))\xrightarrow{\te{Res}}\te{Br}(\mathbb{Q}_2(i,\sqrt[4]{2}))\]
 and hence trivial (cf. proposition \ref{Serr}).

Proposition \ref{Serr} allows us to rewrite the previous equation as

\begin{eqnarray}\label{eqq3}
\te{inv}_u(\mathcal{B}(P))&=& \te{inv}_{u_1}([\frac{-1}{\sqrt{2}}(1-l)^2,(-1+z_0^2+x_0^2-y_0^2)(-1+z_0^2-ix_0^2+iy_0^2)])
\nonumber  
\\
 & = & \te{inv}_{u_1}([\sqrt{2},(-1+z_0^2+x_0^2-y_0^2)(-1+z_0^2-ix_0^2+iy_0^2)]) \nonumber
\\
 & = & \te{inv}_{u_1}([\sqrt{2},(-1+z_0^2-ix_0^2+iy_0^2)]) \nonumber \\
 & = & \te{inv}_{u_2}([\sqrt{2},(-1+z_0^2)^2+(x_0^2-y_0^2)^2])
\end{eqnarray}
Note that to establish the third equality we use (\ref{alll}) and the argument below it.

When $(x,mx,z)\in S$ the following equality holds
\begin{eqnarray}
(-1+z^2)^2+(x^2-y^2)^2&=& \frac{2(z^2-1)(z^2-m^2)}{m^2+1}
\nonumber  
\end{eqnarray}
where $y=mx$. Hence from (\ref{eqq3}) we get the first equality in the following

\begin{eqnarray}\label{eqqq4}
\te{inv}_u(\mathcal{B}(P))&=& \te{inv}_{u_2}([\sqrt{2},(z_0^2-1)(z_0^2-l^2)(l^2+1)]) \nonumber \\
&=& \te{inv}_{u_2}([\sqrt{2},-1-2\sqrt{2}]) \nonumber \\
&=& \te{inv}_{2}([-1,-2])+\te{inv}_{2}([-2,-7]) \nonumber  \\
&=&\frac{1}{2}  \nonumber
\end{eqnarray}

The second equality comes from formulae (\ref{uptosq}). The third from lemma \ref{tocompute}. The fourth from the well-known formulae for inv$_2([a,b])$ (see eg. \cite[XIV, \S 4]{Serre}).

\end{proof}

{\it Proof of theorem \ref{weakoverM}}:
Since $S$ is proper we have that $S(\mathbb{A}_M)=\prod_v S(M_v)$, where the product is taken over all places of $M$.
Let 
$$S(\mathbb{A}_M)^{\mathcal{B}}=\{\{P_v\}\in S(\mathbb{A}_M) |\sum_v\te{inv}_v(\mathcal{B}(P_v))=0  \}.$$
The set $S(\mathbb{A}_M)^{\mathcal{B}}$ is a closed (w.r.t. the adelic or product topology) subset of $S(\mathbb{A}_M)$ that contains $S(M)$ (see \cite[5.2]{Skoro}). Therefore to prove the theorem it suffices to exhibit an adelic point which does not belong to $S(\mathbb{A}_M)^{\mathcal{B}}$.

Define
the adelic point $\{N_v\}$ by: $N_v=Q$ if $v\neq u$ and $N_u=P$.

\noindent Note that since $Q$ is a global point we have that

\[\sum_v\textrm{inv}_v(\mathcal{B}(Q))=0
\]

\noindent Therefore proposition \ref{trul} and lemma \ref{otherpoint} now imply that
\[\sum_v\textrm{inv}_v(\mathcal{B}(N_v))=\textrm{inv}_u(\mathcal{B}(P))-\textrm{inv}_u(\mathcal{B}(Q))=\frac{1}{2}\neq0
\]
ie. $\{N_v\}\notin S(\mathbb{A}_M)^{\mathcal{B}} $, which is what we wanted to show.

\begin{center}
 {\bf Acknowledgement}
\end{center}

\noindent I would like to thank Alexei Skorobogatov for many helpful discussions and guidance throughout 
the course of this research.

\pagebreak

\noindent  Ecole Polytechnique F\'{e}d\'{e}rale de Lausanne, EPFL-SFB-IMB-CSAG, 

\noindent Station 8, CH-1015, Lausanne, Switzerland.

\noindent email address: evis.ieronymou@epfl.ch
\pagebreak

\begin{center}
 APPENDIX
\end{center}
The appendix is to facilitate with computations alluded in the paper (proposition \ref{AA}, theorem \ref{bigthm}, proposition \ref{neothm}).

\noindent The components of the degenerate fibres of the fibration $\phi$.

\begin{center}
\begin{tabular}{|c@{:}c||c@{:}c||c@{:}c|}
\hline
\multicolumn{2}{|c||}{$(0:1)$}  &    \multicolumn{2}{|c||}{$(1:1)$}   & \multicolumn{2}{|c|}{$(-1:1)$}               \\ \hline \hline

$l_{1}$ & $\begin{array}{ccc}x&=&y\\z&=&iw\\\end{array}$ 

&$l_{5}$ &  $\begin{array}{ccc}x&=&z\\y&=&w\\\end{array}$  

&$l_{9}$ & $\begin{array}{ccc}x&=&iz\\y&=&iw\\
\end{array}$                            \\ \hline

$l_{2}$ & $\begin{array}{ccc}x&=&y \\z&=&-iw\\ \end{array}$

&$l_{6}$&  $\begin{array}{ccc}x&=&z \\y&=&-w\\ \end{array}$ 

&  $l_{10}$
 & $\begin{array}{ccc}x&=&iz\\y&=&-iw\\
\end{array}  $  \\ \hline

$l_{3}$ & $\begin{array}{ccc}x&=&-y\\z&=&iw\\\end{array}$

&$l_{7}$ &  $\begin{array}{ccc}x&=&-z\\y&=&w\\\end{array}$

&$l_{11}$ & $\begin{array}{ccc}x&=&-iz\\y&=&iw\\\end{array}$  
  \\ \hline

$l_{4}$ & $\begin{array}{ccc}x&=&-y\\z&=&-iw\\\end{array}$ 

&$l_{8}$ &  $\begin{array}{ccc}x&=&-z\\y&=&-w\\\end{array}$                                 
&$l_{12}$ & $\begin{array}{ccc}x&=&-iz\\y&=&-iw\\
\end{array}$   \\ \hline

\end{tabular}

\bigskip

\begin{tabular}{|c@{:}c||c@{:}c||c@{:}c|}
\hline
\multicolumn{2}{|c||}{$(i:1)$}  &    \multicolumn{2}{|c||}{$(-i:1)$}   & \multicolumn{2}{|c|}{$(1:0)$}               \\ \hline \hline

$l_{13}$ & $\begin{array}{ccc}w&=&\alpha x\\z&=&\alpha y\\\end{array}$ 

&$l_{17}$ &  $\begin{array}{ccc}x&=&\alpha w\\y&=&\alpha z\\\end{array}$  

&$l_{21}$ & $\begin{array}{ccc}x&=&iy\\z&=&w\\
\end{array}$                            \\ \hline

$l_{14}$ & $\begin{array}{ccc}w&=&\alpha x \\z&=&-\alpha y\\ \end{array}$

&$l_{18}$&  $\begin{array}{ccc}x&=&\alpha w \\y&=&-\alpha z\\ \end{array}$ 

&  $l_{22}$
 & $\begin{array}{ccc}x&=&iy\\z&=&-w\\
\end{array}  $  \\ \hline

$l_{15}$ & $\begin{array}{ccc}w&=&-\alpha x\\z&=&\alpha y\\\end{array}$

&$l_{19}$ &  $\begin{array}{ccc}x&=&-\alpha w\\y&=&\alpha z\\\end{array}$

&$l_{23}$ & $\begin{array}{ccc}x&=&-iy\\z&=&w\\\end{array}$  
  \\ \hline

$l_{16}$ & $\begin{array}{ccc}w&=&-\alpha x\\z&=&-\alpha y\\\end{array}$ 

&$l_{20}$ &  $\begin{array}{ccc}x&=&-\alpha w\\y&=&-\alpha z\\\end{array}$                                 
&$l_{24}$ & $\begin{array}{ccc}x&=&-iy\\z&=&-w\\
\end{array}$   \\ \hline

\end{tabular}

\bigskip
where $i=\sqrt{-1}$ and $\alpha^2=i$.
\end{center}

The divisors of $F$ and $G$

\begin{eqnarray}\label{divisor formula}
\te{div}(G) & =& 2l_1+2l_2+l_5+l_8+l_9+l_{12}+l_{13}+l_{16}+l_{17}+l_{20}
\nonumber\\
& &-3l_{21} -3l_{22}-3l_{23}-3l_{24}+\te{horizontal divisors} \\
\te{div}(F) & =& l_1+2l_2+l_4+l_5+l_8+l_9+l_{12} \nonumber\\
&&-3l_{21}-2l_{22}-2l_{23}-l_{24}+\te{horizontal divisors}\nonumber
\end{eqnarray}

In the following table we use affine
coordinates on the open set $w=1$. $\overline s$ denotes the restriction of $s$ to the function field of the corresponding line.

\begin{center}

\begin{tabular}{|c|c|c|c|c|}
\hline
     &     $l_1$                               & $l_2$ &$l_3$&$l_4$   \\  \hline \hline
$G$ 
 &    $\begin{array}{c}
t^2s\\ \overline s=\frac{i}{2y^2}
\end{array}$     
 &
  $\begin{array}{c}
t^2s\\ \overline s=-\frac{i}{2y^2}
\end{array}$
  &
    $\begin{array}{c}
s\\ \overline s=2iy^2
\end{array}$
    & 
     $\begin{array}{c}
s\\ \overline s=-2iy^2
\end{array}$
      \\  \hline  
$F$ 
 &  
$\begin{array}{c}
ts\\ \overline s=\frac{i-1}{y}
\end{array}$ 
     &
      $\begin{array}{c}
t^2s\\ \overline s=\frac{(i-1)y}{2}
\end{array}$
       &   
        $\begin{array}{c}
s\\ \overline s=-2(1-i)y
\end{array}$
        & 
         $\begin{array}{c}
ts\\ \overline s=(1-i)y^3
\end{array}$  \\  \hline  

\end{tabular}

\bigskip
\bigskip

\begin{tabular}{|c|c|c|c|c|}
\hline
     &     $l_{21}$                               & $l_{22}$ &$l_{23}$&$l_{24}$   \\  \hline \hline
$G$ 
 &    $\begin{array}{c}
t^3s\\ \overline s=-1
\end{array}$     
 &
  $\begin{array}{c}
t^3s\\ \overline s=-1
\end{array}$
  &
    $\begin{array}{c}
t^3s\\ \overline s=-1
\end{array}$
    & 
     $\begin{array}{c}
t^3s\\ \overline s=-1
\end{array}$
      \\  \hline

\end{tabular}

\end{center}


\begin{thebibliography}{99}

\bibitem{Perlis} S. An, S. Kim, D. Marshall, S. Marshall, W.  McCallum and
 A. Perlis. Jacobians of genus one curves. J. Number Theory 90 (2001), no. 2, 304-315.

\bibitem{Br2} M. Bright. Brauer groups of diagonal quartic surfaces.  J.
 Symbolic Comput. 41 (2006), no. 5, 544-558.

\bibitem{CTladescente} J-L Colliot-Th\'el\`ene and J-J Sansuc. La descente sur les vari�et�es rationnelles, II. (French) [Descent on rational varieties. II]
Duke Math. J. 54 (1987), no. 2, 375-492.


\bibitem{bigpap} J-L Colliot-Th\'el\`ene, A.N. Skorobogatov and  P. Swinnerton-Dyer. Hasse principle for pencils of curves of genus one whose Jacobians have rational $2$-division points.  Invent. Math.  134  (1998),  no. 3, 579-650.




\bibitem{CT} J-L Colliot-Th\'el\`ene and  P. Swinnerton-Dyer. Hasse principle and weak approximation for pencils of Severi-Brauer and similar varieties.
J. Reine Angew. Math. 453 (1994), 49-112.

\bibitem{Merk} R. Elman, N. Karpenko and A. Merkurjev. The algebraic and geometric theory of quadratic forms. Colloqium Publications, vol. 56,  AMS, Providence, RI.



\bibitem{SerreInv} S. Garibaldi, A. Merkurjev, and J.-P. Serre. Cohomological invariants in Galois cohomology. University Lecture Series, vol. 28, AMS, Providence, RI, 2003.


\bibitem{Groth}  A. Grothendieck. Le groupe de Brauer.  Dix expos\'es sur la cohomologie des sch\'emas, pp. 46-189. North-Holland, Amsterdam (1968).

\bibitem{Har} D. Harari. Obstructions de Manin transcendantes.  Number theory (Paris, 1993--1994), 75-87, 
London Math. Soc. Lecture Note Ser., 235, Cambridge Univ. Press, Cambridge, 1996. 


\bibitem{SkHar} D. Harari and A.N. Skorobogatov. Non-abelian descent and the arithmetic of Enriques surfaces.  Int. Math. Res. Not.  2005,  no. 52, 3203-3228

\bibitem{Morandi} P. Morandi. Field and Galois Theory.  Graduate Texts in Mathematics, No. 167. Springer-Verlag, New York-Heidelberg, 1996

\bibitem{Shaf} I.R. Shafarevich and I.I. Piatetskii-Shapiro. A Torelli theorem for algebraic surfaces of type K3. Math. USSR Izvestija Vol.5 (1971), No3, 547-588




\bibitem{Serre} J.-P. Serre. Local Fields. Graduate Texts in Mathematics, vol. 67, translated from the French by M.J. Greenberg. Springer-Verlag, New York, second edition 1995. 




\bibitem{Silverman} J.H. Silverman. The arithmetic of elliptic curves.  Graduate Texts in Mathematics, 106. Springer-Verlag, New York, 1992. 


\bibitem{Skoro} A.N. Skorobogatov. Torsors and rational points. Cambridge
University Press, Cambridge, 2001.


\bibitem{Skorokaiswin} A.N. Skorobogatov and  P. Swinnerton-Dyer.
$2$-descent on elliptic curves and rational points on certain Kummer surfaces.
Adv. Math. 198 (2005), no. 2, 448-483.

\bibitem{SkoroK3} A.N. Skorobogatov, Y. Zarhin. A finiteness theorem for the Brauer group of abelian varieties and K3 surfaces.  
J. Algebraic Geom. 17 (2008), 481-502.




\bibitem{SD} P. Swinnerton-Dyer. Arithmetic of diagonal quartic surfaces. II.
Proc. London Math. Soc. (3) 80 (2000), no. 3, 513-544.


\bibitem{Witt} O. Wittenberg. Transcendental Brauer-Manin obstruction on a pencil of elliptic curves. Arithmetic of higher-dimensional algebraic varieties (Palo Alto, CA, 2002), 259-267, Progr.
Math., 226, Birkh�auser Boston, Boston, MA, 2004.

\end{thebibliography}
\end{document}